\documentclass[12pt]{article}

\usepackage{amsthm,amssymb,amsmath,fullpage,pstricks}
\usepackage{color}
\usepackage{tikz}
\usetikzlibrary{arrows,automata,positioning}
\usepackage{color}
\RequirePackage[colorlinks,citecolor=blue,urlcolor=blue,linkcolor=blue]{hyperref}

\allowdisplaybreaks

\newcommand{\mc}[1]{\mathcal{#1}}
\newcommand{\indic}[1]{\mathbf{1}_{\{#1\}}}
\newcommand{\R}{\mathbb{R}}
\newcommand{\N}{\mathbb{N}}
\newcommand{\Z}{\mathbb{Z}}

\newcommand{\ee}{\mathrm{e}}

\newcommand{\ra}{\rightarrow}
\renewcommand{\P}{\mathbb{P}}
\newcommand{\E}{\mathbb{E}}

\newcommand{\vep}{\varepsilon}
\newcommand{\nn}{\nonumber}

\newcommand{\blank}[1]{}
\newcommand{\bs}[1]{\boldsymbol{#1}}

\newtheorem{THM}{Theorem}
\newtheorem*{THM*}{Theorem}
\newtheorem{LEM}{Lemma}[section]

\newtheorem{REM}[LEM]{Remark}
\newtheorem{EXA}[LEM]{Example}
\newtheorem{OPEN}{Open Problem}
\newtheorem{COND}{Condition}

\title{\vspace{-40pt}Biased random walk on the trace of\\biased random walk on the trace of \dots}
\author{
{David Croydon\footnote{Department of Advanced Mathematical Sciences, Graduate School of Informatics, Kyoto University. Email:  croydon@acs.i.kyoto-u.ac.jp }}\: and
{Mark Holmes\footnote{School of Mathematics and Statistics, University of Melbourne. Email: holmes.m@unimelb.edu.au}}\;
}

\begin{document}

\maketitle
\abstract{We study the behaviour of a sequence of biased random walks $(X^{(i)})_{i \ge 0}$ on a sequence of random graphs, where the initial graph is $\Z^d$ and otherwise the graph for the $i$th walk is the trace of the $(i-1)$st walk.  The sequence of bias vectors is chosen so that each walk is transient. We prove the aforementioned transience and a law of large numbers, and provide criteria for ballisticity and sub-ballisticity.  We give examples of sequences of biases for which each $(X^{(i)})_{i \ge 1}$ is (transient but) not ballistic, and the limiting graph is an infinite simple (self-avoiding) path.  We also give examples for which each $(X^{(i)})_{i \ge 1}$ is ballistic, but the limiting graph is not a simple path.
}

\tableofcontents

\section{Introduction}

The study of stochastic processes in disordered media is an important aspect of modern probability.  Models in this area for which extensive research has been conducted include the classical model of random walk in random environment, as well as random walks on random graphs, such as Galton-Watson trees and percolation clusters. Typical properties that one is interested in include: (i) recurrence/transience; (ii) laws of large numbers (i.e.\ the existence of a deterministic limiting velocity); (iii) conditions for ballisticity/sub-ballisticity; (iv) regularity (i.e.\ continuity, monotonicity or lack thereof) of attributes (e.g.\ the velocity) in terms of some underlying parameter; and (v) scaling limits.

In this paper, we tackle some of these issues for sequences of biased random walks $(X^{(i)})_{i \ge 0}$ on random graphs, where the initial graph is $\mc{Z}^{(0)}=\Z^d$ and otherwise the graph $\mc{Z}^{(i)}$ for the $i$th walk is the trace of the $(i-1)$st walk (see Figure \ref{fig:paths}). The sequence of biases is chosen so that each walk is transient -- somewhat remarkably, this does not mean that we necessarily require the underlying drift of the $i$th walk to be oriented in the direction of the initial walk, see Remark \ref{counterrem} for an elaboration of this point. By regeneration arguments, which require some care to take into account the multiple processes, we demonstrate the existence of deterministic limiting speeds, see Theorem \ref{thm:main1} below for a precise statement of these results. Regarding the issue of ballisticity, we note that the initial walk $X^{(0)}$, which has non-trivial bias in the direction $e_1$, creates traps for subsequent walks. Moreover, except in trivial settings, the walk $X^{(i)}$ does not visit every site of $\mc{Z}^{(i)}$, so its trace $\mc{Z}^{(i+1)}$ is a strict subset of $\mc{Z}^{(i)}$, meaning that the walk $X^{(i)}$ may delete some traps and create new ones. We will show that the effect of trapping can lead to zero speeds, and in particular establish a sharp phase transition for whether the walk $X^{(i)}$ is ballistic or sub-ballistic, see Theorem \ref{thm:main2} below. Finally, we exhibit conditions depending on the sequence of biases that ensure the limiting graph is or is not an infinite simple path, see Theorem \ref{thm:main3}. In certain cases (such as when the sequence of biases is constant) the law of this limit may be of independent interest.

Before introducing our model and results, let us briefly relate our work to other studies in which trapping has been observed for biased random walk on random graphs.  As early as the 1980s, physicists observed that such phenomenon might be relevant when the random graphs are percolation clusters, empirically demonstrating the non-monotonicity of the speed, and sub-ballisticity in the strong bias regime \cite{BD}. Mathematically, a phase transition between ballisticity and sub-ballisticity was first shown rigorously for the simpler model of random walk on supercritical Galton-Watson trees \cite{LPP} (see also \cite{BFGH,Bowditch,CFK} for recent work concerning more detailed properties of such processes), and has since been confirmed to hold in the percolation setting \cite{BGP, FH, Sznitman}. A relatively up-to-date survey of these developments is given in \cite{BAF}. Qualitatively, our results match those established for Galton-Watson trees and percolation clusters, and, although we do not confirm it rigorously, we also observe empirically non-monotonic behaviour for the speed that is similar to the behaviour expected for these other models. Moreover, whilst our graphs are more complex than trees, in the sense there is not a unique shortest path between vertices and the traps are less obviously defined, the model is still more tractable than the percolation case. As a result, we are able to give a more concrete expression for the critical point that separates the ballistic and sub-ballistic phase, which we are even able to evaluate explicitly in examples. Finally, we note that, in another related work, biased random walk on an unbiased random walk has been shown to exhibit localisation on a logarithmic scale \cite{C2}.

\begin{figure}
\begin{center}
\vspace{-30pt}
\includegraphics[angle=270,width=0.9\textwidth]{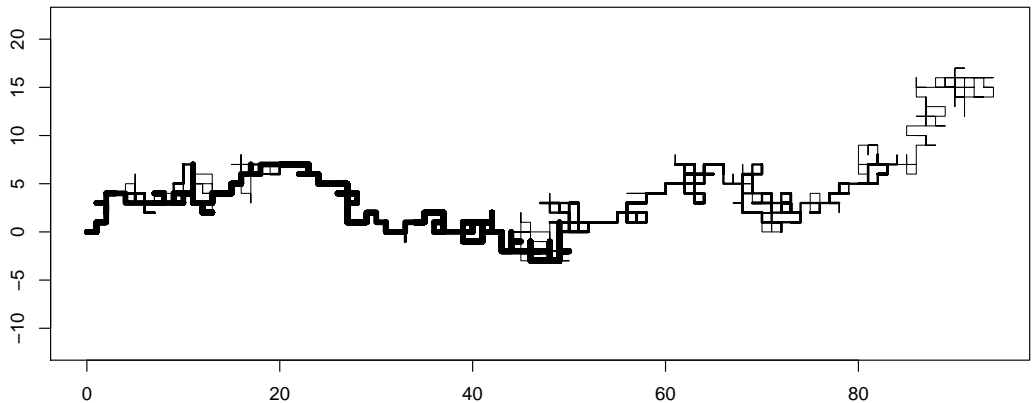}
\vspace{-40pt}

\end{center}
\caption{A simulation of a small finite piece of the walks $X^{(0)},X^{(1)},X^{(2)}$ (shaded thinnest/lightest to thickest/darkest), in 2 dimensions with common biases $p^{(i)}(e_1)=3/9$, and $p^{(i)}(e)=2/9$ for $e\ne e_1$.}
\label{fig:paths}
\end{figure}

\subsection{The model}

Let $(\bs{p}^{(i)})_{i \in \Z_+\cup\{\infty\}}$ be a sequence of probability distributions on the standard basis vectors $\{\pm e_j:j \in [d]\}$ in $\Z^d$, where $[d]:=\{1,2,\dots,d\}$ and $d\ge 2$. Given such a probability distribution $\bs{p}$, and a connected subgraph $\mc{Z}$ of $\Z^d$ that contains the origin $0\in\mathbb{Z}^d$, we define a $\bs{p}$-random walk $X=(X_n)_{n \in \Z_+}$ on $\mc{Z}$ to be the discrete-time Markov chain starting at $0$ with transition probabilities
\[P^{\mc{Z}}(X_{n+1}=x+e|X_n=x)=\begin{cases}
\frac{p(e)}{\sum_{e':(x,x+e')\in \mc{Z}}p(e')}, & \textrm{ if }(x,x+e)\in \mc{Z},\\
0, & \textrm{ otherwise}.
\end{cases}\]
Let $\mc{Z}^{(0)}=\Z^d$, and let $X^{(0)}$ be a $\bs{p}^{(0)}$-random walk on $\mc{Z}^{(0)}$, i.e.~a simple random walk on $\Z^d$ with step distribution $\bs{p}^{(0)}$.  Given a walk $X^{(i)}$ taking values in $\Z^d$, we let $\mc{Z}^{(i+1)}=(\mc{V}^{(i+1)},\mc{E}^{(i+1)})$ be the connected graph with vertex set \[\mc{V}^{(i+1)}=\left\{X^{(i)}_n:n \in \Z_+\right\},\]
and (undirected) edge set
\[\mc{E}^{(i+1)}=\left\{(X^{(i)}_n,X^{(i)}_{n+1}):n \in \Z_+\right\}.\]
If $\bs{p}^{(0)}$ is {\em biased} (${p}^{(0)}(e)\ne {p}^{(0)}(-e)$ for some $e$) and {\em non-trivial} (${p}^{(0)}(e)\neq 1$ for any $e$), then the walk $X^{(0)}$ is transient (in fact ballistic) and the graph $\mc{Z}^{(1)}$ is random.

The above procedure can now be iterated, leading to a sequence of walks $(X^{(i)})_{i \in \Z_+}$ (with $X^{(i)}$ being a $\bs{p}^{(i)}$-random walk on $\mc{Z}^{(i)}$ for each $i$) that are conditionally independent given the graphs $\mc{Z}^{(i)}$ (that is, given $\mc{Z}^{(i)}$, $X^{(i)}$ is conditionally independent of $X^{(j)}$ for $j<i$).  The sequence of graphs is decreasing ($\mc{Z}^{(i+1)}\subseteq \mc{Z}^{(i)}$ for each $i$), and as such we can define
\[\mc{Z}^{(\infty)}:=\bigcap_{i \in \Z_+}\mc{Z}^{(0)}.\]
We denote by $X^{(\infty)}$ a $\bs{p}^{(\infty)}$-random walk on $\mc{Z}^{(\infty)}$.  We will suppose the sequence of walks $(X^{(i)})_{i\in\Z_+\cup\{\infty\}}$ is defined on a probability space $(\Omega, \mc{F}, \P)$.

Let $\delta^{(i)}_j=p^{(i)}(e_j)-p^{(i)}(-e_j)$, and $\delta^{(i)}=(\delta^{(i)}_j)_{j \in [d]}$. This vector represents the drift of a $\bs{p}^{(i)}$-random walk on $\mathbb{Z}^d$. In particular, $\mathbb{P}$-a.s.,
\[v^{(0)}:=\lim_{n \ra \infty}n^{-1}X^{(0)}_n=\delta^{(0)}.\]
Moreover, under the assumption that $p^{(i)}(e)>0$ for every $e$, let $\ell^{(i)}=\hat{\ell}^{(i)}/\|\hat{\ell}^{(i)}\|$, where $\|\cdot \|$ denotes the Euclidean norm, and $\hat{\ell}^{(i)}=(\hat{\ell}^{(i)}_j)_{j \in [d]}\in \R^d$ is defined by
\[\hat{\ell}^{(i)}_j=\log\left(\frac{p^{(i)}(e_j)}{p^{(i)}(-e_j)}\right).\]
In general, $\ell^{(i)}\ne \delta^{(i)}/\|\delta^{(i)}\|$.  However, if $\delta^{(i)}\ne 0$ then
\begin{align*}
\ell^{(i)}\cdot \delta^{(i)}
&=\frac{1}{\|\hat{\ell}^{(i)}\|}\sum_{j=1}^d [\log p^{(i)}(e_j) - \log p^{(i)}(-e_j)][p^{(i)}(e_j) - p^{(i)}(-e_j)]>0
\end{align*}
since all of the summands are non-negative and at least one is strictly positive.

We henceforth assume the following.
\begin{COND}
\label{cond:1}
The $(\bs{p}^{(i)})_{i \in \Z_+\cup \{\infty\}}$ are such that
\begin{itemize}
\item[(a)] $\delta^{(0)}_j\ge 0$ for each $j\in [d]$, and $\delta^{(0)}_1>0$, and
\item[(b)] $p^{(i)}(e)>0$ for all $e \in \{\pm e_j:j \in [d]\}$, and
\item[(c)] for each $i\in\N\cup\{\infty\}$, $\delta^{(0)}\cdot \ell^{(i)}>0$.
\end{itemize}
\end{COND}

We lose no generality in assuming Condition \ref{cond:1}(a); it is included solely for the purpose of fixing a direction of transience for the walk $X^{(0)}$. Condition \ref{cond:1}(b) ensures that the walks always have an available move. Condition \ref{cond:1}(c) is the condition required to ensure that all subsequent walks are also transient.  We highlight that it involves the vector $\ell^{(i)}$ rather than $\delta^{(i)}$ (the importance of this distinction is discussed in Remark \ref{counterrem}).  Note that, given the other two conditions, it is also sharp (see Lemma \ref{lem:rec}).

A further advantage to assuming Condition \ref{cond:1}(b) is that it allows us to express the laws of the processes in terms of conductance networks. More precisely, for each $i\in\Z_+\cup\{\infty\}$, let
\[c_{(i),j}=p^{(i)}(-e_j),\qquad\beta_{(i)}=\exp\{\|\hat{\ell}^{(i)}\|\}\ge 1.\]
Then for each $x=(x_1,x_2,\dots,x_d),y=(y_1,y_2,\dots,y_d)\in\mathbb{Z}^d$ with $\|x-y\|=1$, define
\begin{equation}\label{conductances}
c^{(i)}(x,y)=\left(\prod_{j=1}^d c_{(i),j}^{|y_j-x_j|}\right)\beta_{(i)}^{(x\vee y)\cdot\ell^{(i)}},
\end{equation}
where $x\vee y =(x_j\vee y_j)_{j\in [d]}$.  The quantity $c^{(i)}(x,y)$ is called the {\em conductance} of edge $(x,y)$ for $\bs{p}^{(i)}$.  Thus $\ell^{(i)}$ (which is non-zero under Condition \ref{cond:1}) describes the direction in which the conductances grow most rapidly, and $\beta_{(i)}$ (which is strictly greater than one under Condition \ref{cond:1}) describes the rate of increase.  Let $c^{(i)}(x)=\sum_{y:(x,y)\in \mc{E}^{(i)}} c^{(i)}(x,y)$, and, to simplify notation, write $P^{(i)}$ for $P^{\mc{Z}^{(i)}}$.  It is straightforward to check that, given $\mc{Z}^{(i)}$
\begin{equation}
P^{(i)}(X^{(i)}_{n+1}=x+e|X^{(i)}_n=x)=\begin{cases}
\frac{c^{(i)}(x,x+e)}{c^{(i)}(x)}, & \textrm{ if }(x,x+e)\in \mc{Z}^{(i)},\\
0, & \textrm{ otherwise}.\nn
\end{cases}
\end{equation}

\subsection{Main results}

We now introduce the main results that were briefly outlined above. Firstly, we establish directional transience and existence of limiting velocities.

\begin{THM}\label{thm:main1}
Assume Condition \ref{cond:1}.  It then
$\P$-a.s.\ holds that:
\begin{itemize}
\item[(a)] for each $i\in \mathbb{Z}_+\cup\{\infty\}$, $\lim_{n\rightarrow\infty}X^{(i)}_n\cdot \ell \ra \infty$ for every $\ell$ such that $\ell\cdot \delta^{(0)}>0$;
\item[(b)] for each $i\in \mathbb{Z}_+$, there exists a deterministic $\kappa_{(i)}\in[0,\infty)$ such that
\[v^{(i)}:=\lim_{n \ra \infty} n^{-1}X^{(i)}_n=\kappa_{(i)}\delta^{(0)}.\]
\end{itemize}
\end{THM}

\begin{REM}
Condition \ref{cond:1} is sufficient, but not necessary, to obtain the conclusions of Theorem \ref{thm:main1}.  Indeed, if $d=2$, and, for each $i$, $p^{(i)}(e_1)>0$ and $p^{(i)}(-e_1)=0$ (i.e.~$-e_1$ is a forbidden direction), and $p^{(i)}(e_2), p^{(i)}(-e_2)>0$, then Condition \ref{cond:1}(b) is violated but the conclusions of Theorem \ref{thm:main1} are relatively simple to obtain.
\end{REM}

Our second main result concerns conditions for ballisticity/sub-ballisticity. To state this, we need to introduce some further notation. Let
\begin{equation}\label{varphidef}
\varphi_{(i)}(t)=\E\left[\exp\{-tX^{(0)}_1\cdot \ell^{(i)}\}\right].
\end{equation}
It is then the case that there is a unique positive solution $t_{(i)}$ to $\varphi_{(i)}(t)=1$ (see Lemma \ref{backtrackprob} below), and we set
\[\alpha_{(i)}=\exp\{t_{(i)}\}.\]

\begin{THM}\label{thm:main2}
Assume Condition \ref{cond:1}.   Then for each $i\in \mathbb{N}$:
\begin{itemize}
\item[(a)] If $\beta_{(i)}<\alpha_{(i)}$, then $X^{(i)}$ is ballistic, i.e.\ $v^{(i)}\neq 0$.
\item[(b)] If $\beta_{(i)}>\alpha_{(i)}$,  then $X^{(i)}$ is sub-ballistic, i.e.\ $v^{(i)}=0$.
\end{itemize}
\end{THM}
Figure \ref{fig:bad} shows simulations of $v^{(1)}\cdot e_1$ for walks on $\Z^2$ where $p^{(0)}(e_1)=2/5$ and $p^{(0)}(e)=1/5$ otherwise, while $p^{(1)}(e_1)=r/(r+3)$ and $p^{(0)}(e)=1/(r+3)$ otherwise, as a function of $r \in [1,2.5]$.  Theorem \ref{thm:main2} shows that $v^{(1)}\cdot e_1=0$ when $r>2$ in this case.  While non-monotonicity is supported by the figure, the variability between realisations shows that $5\times 10^7$-step walks are insufficient to identify a phase transition by simulation.
\begin{figure}
\begin{center}
\vspace{-10pt}
\includegraphics[scale=.5]{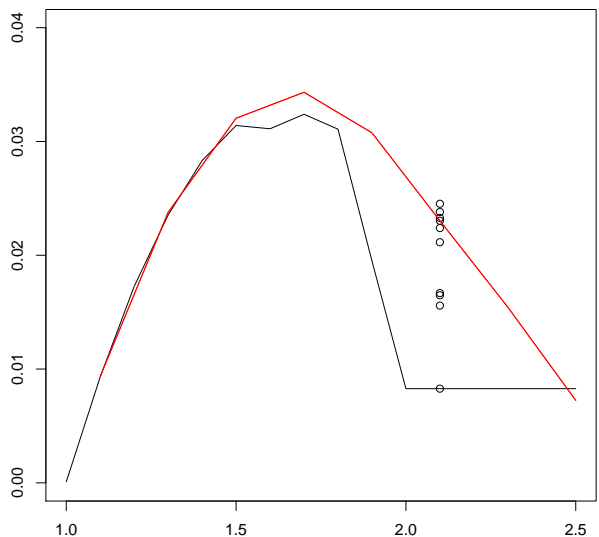}
\vspace{-40pt}
\end{center}
\caption{Two simulations of $v^{(1)}\cdot e_1$ (for walks on $\Z^2$ where $p^{(0)}(e_1)=2/5$ and $p^{(0)}(e)=1/5$ otherwise, and $p^{(1)}(e_1)=r/(r+3)$ and $p^{(0)}(e)=1/(r+3)$ otherwise), as a function of $r \in [1,2.5]$. The plot also shows 10 simulations of $v^{(1)}\cdot e_1$ when $r=2.1$.   Each point in the graph is calculated based on the endpoint of a $5\times 10^7$-step $p^{(1)}$-walk (on the trace of a $p^{(0)}$-walk).}
\label{fig:bad}
\end{figure}

As noted above, one interesting feature of our result is that we obtain an explicit representation for the critical point for our model defined on $\Z^d$. As we will describe in detail in Sections \ref{balsec} and \ref{subsec}, the parameter $\alpha_{(i)}$ is related to the decay rate of the probability the random walk $X^{(0)}$ backtracks a certain distance in direction $\ell^{(i)}$, an event which can potentially cause a trap for the $i$th walk. For further intuition about the relevance of this parameter, see the discussion prior to the proof of Lemma \ref{trapprob}, and the detailed construction of a trap for $X^{(i)}$ within the proof. To illustrate the simplicity of the condition in Theorem \ref{thm:main2}, we next present a canonical example, in which it is even possible to compute the parameters $\alpha_{(i)}$ analytically.

\begin{EXA}
\label{exa:canonical}
Suppose that for each $i\geq 0$ there exists a $k_i\in [d]$ and $\gamma_i>1$ such that
\[p^{(i)}(e)= \frac{1+(\gamma_i-1) \indic{e\in\{e_1,\dots,e_{k_i}\}}}{2d+k_i(\gamma_i-1)}.\]
We then have that
\[\delta^{(i)}_j=\frac{\gamma_i-1}{2d+k_i(\gamma_i-1)}\indic{j\le k_i}, \qquad \hat{\ell}^{(i)}_j=\log(\gamma_i)\indic{j\le k_i},\]
and the conductances are given by (\ref{conductances}) with
\[c_{(i),j}=1, \qquad \log \beta_{(i)}=\sqrt{k_i}\log\gamma_i.\]
In this case, it is straightforward to check that Condition \ref{cond:1} holds, and deduce that
\[v^{(0)}:=\lim_{n \ra \infty} \frac{X^{(0)}_n}{n}=\delta^{(0)}=\frac{\gamma_0-1}{2d+k_0(\gamma_0-1)}(e_1+e_2+\dots+e_{k_0}).\]
We moreover observe that the value of ${\alpha}_{(i)}$ can be computed explicitly, thus yielding completely transparent criteria for ballisticity/sub-ballisticity. Indeed, the equation $\varphi_{(i)}(t)=1$ can be rewritten
\begin{align*}
&{k_i\ee^{t/\sqrt{k_i}}+\ee^{-t/\sqrt{k_i}}\left(k_i+(k_0\wedge k_i)(\gamma_0-1)\right)+2\left(d-(k_0\vee k_i)\right)+\left((k_0\vee k_i) -k_i\right)(\gamma_0+1)}\\
&\hspace{300pt}=2d+k_0(\gamma_0-1).
\end{align*}
Multiplying up by $\ee^{t/\sqrt{k_i}}$, gives a quadratic in $\ee^{t/\sqrt{k_i}}$, which is readily solved to give that
\[\log\alpha_{(i)}=t_{(i)}=\sqrt{k_i}\log\left(1+\frac{k_i\wedge k_0}{k_i}\left(\gamma_0-1\right)\right).\]
(The other root of the quadratic is obtained by setting $t=0$.) Thus, the $i$th walk is ballistic if
\begin{equation}\label{gammacond}
k_i(\gamma_i-1)<(k_i\wedge k_0)(\gamma_0-1),
\end{equation}
and sub-ballistic if the reverse (strict) inequality is true. In particular, if $k_i\leq k_0$, then the condition \eqref{gammacond} simplifies further to $\gamma_i<\gamma_0$.
\end{EXA}

Thus far our results have depended only on relations between $\bs{p}^{(i)}$ and $\bs{p}^{(0)}$ for individual $i$.  Our final main result concerns the limiting graph $\mc{Z}^{(\infty)}$, and depends on the entire sequence $(\bs{p}^{(i)})_{i \in \Z_+}$.

\begin{THM}\label{thm:main3}
Assume Condition \ref{cond:1}.
\begin{itemize}
\item[(a)] If there exists some $\bs{p}^{(i)}$ such that $\bs{p}^{(i')}=\bs{p}^{(i)}$ for infinitely many $i'$, then $\mc{Z}^{(\infty)}$ is almost-surely a simple path.
\item[(b)] If for some $e\ne e_1$, and all $c\in(0,\infty)$,
\[\sum_{i=1}^\infty \frac{c^{(i)}(0,e_1)}{c^{(i)}(0,e)}\min\left\{1,(\beta_{(i)}^{c\delta^{(0)}\cdot \ell^{(i)}}-1)\beta_{(i)}^c\right\}<\infty,\]
then $\P(\mc{Z}^{(\infty)} \text{ is not simple})>0$.
\end{itemize}
\end{THM}

The following is an example of a sequence for which all but finitely many $X^{(i)}$ are ballistic in direction $e_1$, but $\mc{Z}^{(\infty)}$ is almost-surely not a simple path. Note that the random walks become increasingly symmetric as $i\rightarrow\infty$, and as a result spend increasingly long times in finite sections of the graph.

\begin{EXA}
Let $d=2$, and $\bs{p}^{(0)}$ be chosen so that $\delta^{(0)}=\frac{1}{2} e_1$. Set \[p^{(i)}(e_1)=\frac14+\vep_i,\qquad p^{(i)}(-e_1)=\frac14-\vep_i,\qquad p^{(i)}(e_2)=p^{(i)}(-e_2)=\frac14,\]
where $\varepsilon_i\in(0,\frac14)$ for each $i$, and $\sum_{i=1}^{\infty}\varepsilon_i<\infty$. Observe that $\ell^{(i)}=e_1$ (for each $i$) since
\[\hat{\ell}^{(i)}=\left(\log((1+4\vep_i)/(1-4\vep_i)),0\right),\]
so $\alpha_{(i)}=\alpha>1$.  Moreover, $\beta_{(i)}=(1+4\vep_i)/(1-4\vep_i)\to 1$, so Theorem \ref{thm:main2}(a) applies for all but finitely many $i$.  With $e=e_2$, the summand in Theorem \ref{thm:main3}(b) is equal to
\[4\left(\frac14+\vep_i\right)\left(\left(\frac{1+4\vep_i}{1-4\vep_i}\right)^{c/2}-1\right)\left(\frac{1+4\vep_i}{1-4\vep_i}\right)^{c},\]
which is asymptotic to $4c\varepsilon_i$ as $i\rightarrow\infty$, and hence Theorem \ref{thm:main3}(b) applies.
\end{EXA}

In our second example, we illustrate an alternative way in which it might transpire that $\mathcal{Z}^{(\infty)}$ is not simple. In particular, this is when we have an increasingly large drift into traps.

\begin{EXA}
Let $d=2$, and $\bs{p}^{(0)}$ be chosen so that $\delta^{(0)}=\frac{1}{2} e_1$. Set
\[p^{(i)}(e_1)=\frac{1}{2}\varepsilon_i,\qquad p^{(i)}(-e_1)=\frac{1}{4}\varepsilon_i,\qquad p^{(i)}(-e_2)=\frac{1}{4}\varepsilon_i,\qquad p^{(i)}(e_2)=1-\varepsilon_i,\]
where $\varepsilon_i\in(0,1)$ for each $i$, and $\sum_{i=1}^{\infty}\varepsilon_i<\infty$.  Here, $\ell^{(i)}\to e_2$, $\beta_{(i)} \to \infty$ and $\alpha_{(i)}\to \alpha\geq1$ as $i\to \infty$, but $\ell^{(i)}\cdot \delta^{(0)}>0$ for each $i$.  Thus, all walks are transient and all but finitely many are sub-ballistic. We then have that, again with $e=e_2$, the summand in Theorem \ref{thm:main3}(b) is bounded above by
\[\frac{\vep_i}{2(1-\vep_i)},\]
which is summable.
\end{EXA}

\subsection{Open problems}

We now collect some open problems that arise from the present work.

\begin{OPEN}
The regeneration techniques we apply to deduce the existence of velocities in Theorem \ref{thm:main1} do not immediately extend to the $i=\infty$ case.  Prove the existence of a deterministic velocity in this case.
\end{OPEN}

\begin{OPEN}
Theorem \ref{thm:main2} gives a phase transition for ballisticity, but it does not resolve the behaviour of the walk at the phase transition. If we assume Condition \ref{cond:1}, then is it the case that $X^{(i)}$ is sub-ballistic when $\beta_{(i)}=\alpha_{(i)}$?
\end{OPEN}

\begin{OPEN}
As has been studied for Galton-Watson trees and percolation clusters, it is natural to consider the scaling limit of the walk $X^{(i)}$. In particular, we would expect that the asymptotic behaviour of the process is determined by the parameter $\gamma_{(i)}:=\min\{2,\alpha_{(i)}/\beta_{(i)}\}$. In particular, if $\gamma_{(i)}>2$, then we would expect to see a central limit theorem for $X^{(i)}\cdot\delta^{(0)}$ around its mean behaviour. For $\gamma_{(i)}\in(1,2)$, then we would expect to see fluctuations related to a $\gamma_{(i)}$-stable distribution around the mean, and for $\gamma_{(i)}\in(0,1)$ we would expect that $X^{(i)}\cdot\delta^{(0)}$ behaves like the inverse of a $\gamma_{(i)}$-stable subordinator. Moreover, in the latter cases, due to a lattice effect that is also seen in other models, it might be anticipated that the relevant scaling limits only hold along subsequences. See \cite{Bowditch} for recent progress in the Galton-Watson tree case, and \cite{BAF} for a discussion of results and conjectures in the percolation case.
\end{OPEN}

\begin{OPEN}
How does the speed $v^{(i)}$ vary as a function of $\bs{p}^{(i)}$? Is it continuous? Where is it maximised? Is it non-monotonic and unimodal as a function of $\beta_{(i)}$ (for other parameters fixed)?
\end{OPEN}

\begin{OPEN} Perpendicular to the previous question, one might consider how $v^{(i)}$ varies with $i$. To give one concrete question in this direction, suppose that $X^{(1)}$ is ballistic, and set $\bs{p}^{(i)}=\bs{p}^{(1)}$ for every $i\ge 2$. Is it the case that the speeds are increasing in $i$?
\end{OPEN}

\begin{OPEN} We believe that solving these problems would be simpler on trees. To what extent is it possible to answer any of the above questions if the original walk has as its state space a regular tree or a supercritical Galton-Watson tree conditioned to survive, for example?
\end{OPEN}

\begin{OPEN} Theorem \ref{thm:main3} gives sufficient criteria for $\mathcal{Z}^{(\infty)}$ almost-surely being a simple graph, or this not being the case. Is it possible to give a sharp criterion for simplicity of $\mc{Z}^{(\infty)}$ in terms of basic properties of the transition probabilities $\bs{p}^{(i)}$, $i\in\mathbb{Z}_+$?
\end{OPEN}

\begin{OPEN}
If $\bs{p}^{(i)}=\bs{p}^{(0)}$ for every $i$, then $\mathcal{Z}^{(\infty)}$ is a simple path by Theorem \ref{thm:main3}(a).  Is there a relationship between $\mathcal{Z}^{(\infty)}$ and the law of a $\bs{p}^{(0)}$ loop-erased random walk?
\end{OPEN}

\subsection{Organisation and notational conventions}

The remainder of the article is organised as follows. In Section \ref{transsec}, we characterise recurrence and transience, proving Theorem \ref{thm:main1}(a) in particular. This is followed in Section \ref{sec:velocity} by a proof of Theorem \ref{thm:main1}(b), namely the existence of deterministic velocities, which utilises a particular regeneration structure for the multiple walks. Sections \ref{balsec} and \ref{subsec} contain the proof of Theorem \ref{thm:main2}, with the first of these sections describing the ballistic regime, and the second the sub-ballistic one. Finally, in Section \ref{zinftysec}, we study the graph $\mathcal{Z}^{(\infty)}$.

Regarding notational conventions, $c$ and $C$ are constants that may change from line-to-line.

\section{Transience}\label{transsec}

The aim of this section is to prove Theorem \ref{thm:main1}(a), that is, we will demonstrate transience of the random walks $X^{(i)}$ under Condition \ref{cond:1}. In the proof, we write $R^{(i)}$ for the effective resistance metric on $\mathcal{Z}^{(i)}$ when conductances are given by \eqref{conductances}. (For background on effective resistance, see \cite[Chapter 2]{Barlow} or \cite[Chapter 2]{LP}, for example.) We also write
\[r^{(i)}(x,y)=\frac{1}{c^{(i)}(x,y)}\]
for the corresponding edge resistances.

\begin{proof}[{\bf \em Proof of Theorem \ref{thm:main1}(a)}] From the law of large numbers, we $\P$-a.s.\ have that
\begin{equation}\label{lln}
\lim_{n\rightarrow\infty}\frac{X^{(0)}_n}{n}=\E[X_1^{(0)}]=\delta^{(0)}.
\end{equation}
It follows that for any nearest-neighbour path $\pi=(\pi_n)_{n\geq 0}$ with $\pi_0=0$, $\pi_n\in \mathcal{V}^{(1)}$ and $|\pi_n|\rightarrow\infty$,
\begin{equation}\label{pitrans}
\lim_{n\rightarrow \infty}\pi_n\cdot \ell=\infty,
\end{equation}
for every $\ell$ such that $\ell\cdot\delta^{(0)}>0$, i.e.\ $\pi$ is transient in the direction $\ell$. Moreover, taking the supremum over such paths that are also injective, we have for any $i\in\mathbb{N}\cup\{\infty\}$ that (using \eqref{conductances}),
\[\sup_{\pi}\sum_{n=0}^\infty r^{(i)}\left(\pi_n,\pi_{n+1}\right)\leq
\sum_{n=0}^\infty r^{(i)}\left(X_n^{(0)},X_{n+1}^{(0)}\right)\leq C_{(i)}\sum_{n=0}^\infty \beta_{(i)}^{-X^{(0)}_n\cdot \ell^{(i)}}\]
where $C_{(i)}$ is a deterministic constant. Now, by \eqref{lln} and the assumption that $\delta^{(0)}\cdot\ell^{(i)}>0$, we have that, $\P$-a.s., for large $n$,
\[\beta_{(i)}^{-X^{(0)}_n\cdot \ell^{(i)}}\leq \beta_{(i)}^{-n\delta^{(0)}\cdot \ell^{(i)}/2},\]
and so we conclude that, $\P$-a.s.,
\begin{equation}\label{pisum}
\sup_{\pi}\sum_{n=0}^\infty r^{(i)}\left(\pi_n,\pi_{n+1}\right)<\infty.
\end{equation}

We use the preceding deductions as the basis of an inductive argument. To begin with, observe that Rayleigh's monotonicity principle \cite[Chapter 2]{LP} implies that for any injective path $\pi$ of the form considered in the previous paragraph, we have that, $\P$-a.s.,
\[R^{(1)}(0,\infty)\leq \sum_{n=0}^\infty r^{(1)}\left(\pi_n,\pi_{n+1}\right)<\infty.\]
Hence, by \cite[Theorem 2.11]{Barlow}, $\mc{Z}^{(1)}$ is $\P$-a.s.\ a $\bs{p}^{(1)}$-transient graph. Suppose now that we have shown that $\mc{Z}^{(i)}$ is $\P$-a.s.\ a $\bs{p}^{(i)}$-transient graph for some $i\geq 1$. By definition, this implies that $|X^{(i)}_n|\rightarrow \infty$, $\P$-a.s., and so $\mc{Z}^{(i+1)}$ contains at least one injective path $\pi$ from $0$ to $\infty$ with vertices in $\mathcal{V}^{(i)}\subseteq \mathcal{V}^{(1)}$, $\P$-a.s. Thus, from \eqref{pisum}, $\P$-a.s.,
\[R^{(i+1)}(0,\infty)\leq \sum_{n=0}^\infty r^{(i+1)}\left(\pi_n,\pi_{n+1}\right)<\infty,\]
and so $\mc{Z}^{(i+1)}$ is a $\bs{p}^{(i+1)}$-transient graph. We have therefore established that for each $i\in \N$, $\mc{Z}^{(i)}$ is a $\bs{p}^{(i)}$-transient graph $\P$-a.s.

We now deal with the $i=\infty$ case. In particular, it must hold that $\mc{Z}^{(\infty)}$ is an infinite graph, $\P$-a.s. Indeed, if this is not the case, then it must hold that there is a strictly positive probability that one of the graphs $\mc{Z}^{(i)}$, $i\in\mathbb{Z}_+$ is finite, and this clearly contradicts the conclusion of the previous paragraph. As a consequence, we know that, $\P$-a.s., $\mc{Z}^{(\infty)}$ contains at least one injective path $\pi$ from $0$ to $\infty$ with vertices in $\mathcal{V}^{(1)}$, and we can use our previous argument to deduce its transience.

To complete the proof, we are required to check transience in the direction $\ell$. To this end, we observe that the transience we have already established yields that $|X^{(i)}_n|\rightarrow \infty$ for every $i\in \mathbb{Z}_+\cup\{\infty\}$, $\P$-a.s. Appealing to \eqref{pitrans} (and also \eqref{lln}), we readily see that this implies $X^{(i)}_n\cdot \ell\rightarrow \infty$ for every $i\in\mathbb{Z}_+\cup\{\infty\}$ and $\ell$ such that $\ell\cdot\delta^{(0)}>0$, $\P$-a.s.
\end{proof}

By a minor adaptation of the previous argument, it is possible to check, as we do in Lemma \ref{lem:rec} below, that the condition $\delta^{(0)}\cdot\ell^{(i)}>0$ is in fact necessary for transience. To proceed with this, and for later in the paper, it will be helpful to introduce a regeneration structure for $X^{(0)}$ in the direction $\ell$, where $\ell$ can be chosen to be any vector satisfying $\ell\cdot\delta^{(0)}>0$. More precisely, let $\mathcal{T}^{\ell}_j$ denote the $j$th regeneration time for $X^{(0)}$ in the direction $\ell$, i.e.\ the $j$th time that $X^{(0)}\cdot\ell$ hits a new maximum, from which it does not return to a lower level in the future. We then have the following standard result.

\begin{LEM}\label{regenlem} Assume the parts of Condition \ref{cond:1} regarding $\bs{p}^{(0)}$. For any $\ell$ satisfying $\ell\cdot\delta^{(0)}>0$, the following statements hold.\\
(i) $\mathbb{P}$-a.s., the regeneration times $(\mathcal{T}^{\ell}_j)_{j=1}^\infty$ are all finite.\\
(ii) The path segments
\[\left(X^{(0)}_{\mathcal{T}^{\ell}_j+n}-X^{(0)}_{\mathcal{T}^{\ell}_j}\right)_{n=0}^{\mathcal{T}^{\ell}_{j+1}-\mathcal{T}^{\ell}_{j}},\qquad j\geq 1,\]
are independent and identically distributed.\\
(iii) $\mathbb{P}$-a.s., it holds that
\[\lim_{j\rightarrow\infty}\frac{\mathcal{T}_j^{\ell}}{j}=C,\]
where $C\in[1,\infty)$ is a deterministic constant (depending on $i$).
\end{LEM}
\begin{proof} Note that the Markov property and the directional transience of $X^{(0)}$ imply that once $X^{(0)}\cdot \ell$ hits a new maximum, there is strictly positive probability that it does not return to a lower level in the future. Part (i) of the lemma readily follows. The remaining parts of the lemma are obtained by an application of results from \cite{SZ}. Specifically, part (ii) is a consequence of \cite[Theorem 1.4]{SZ}. Moreover, Condition \ref{cond:1} ensures that the uniform ellipticity condition (0.1) and the drift condition (2.37) of \cite{SZ} hold for $X^{(0)}$ with respect to the direction $\ell$. By \cite[Proposition 2.4]{SZ}, this gives us that Kalikow's condition \cite[(0.7)]{SZ} holds, and so we can apply equation \cite[(2.32)]{SZ} to deduce that, $\mathbb{P}$-a.s.,
\[\lim_{j\rightarrow\infty}\frac{X^{(0)}_{\mathcal{T}_j^{\ell}}\cdot\ell}{j}=C,\]
for some finite deterministic constant $C$. Together with the almost-sure result that $n^{-1}X^{(0)}_n\rightarrow \delta^{(0)}$, part (iii) of the lemma follows.
\end{proof}

We now confirm that if $\delta^{(0)}\cdot\ell^{(i)}\leq 0$, then the $i$th walk is recurrent. Since this is not central to the main results of this paper, we are somewhat brief in the proof.

\begin{LEM}\label{lem:rec} Assume Condition \ref{cond:1}(a) and \ref{cond:1}(b) hold, but Condition \ref{cond:1}(c) does not, i.e.\ $\delta^{(0)}\cdot\ell^{(i)}\le 0$ for some $i\in\mathbb{N}\cup\{\infty\}$. It then holds that the random walk $X^{(i)}$ is recurrent.
\end{LEM}
\begin{proof} First suppose that $\delta^{(0)}\cdot \ell^{(i)}<0$.  Let $(\mathcal{T}_j)_{j\geq 1}$ be the regeneration times for $X^{(0)}$ in direction $\delta^{(0)}$. It then holds that
\begin{equation}\label{rinfty}
R^{(i)}(0,\infty)\geq \sum_{j=2}^{\infty}r^{(i)}\left(X^{(0)}_{\mathcal{T}_j-1},X^{(0)}_{\mathcal{T}_j}\right)
\geq C \sum_{j=2}^{\infty}\beta_{(i)}^{-X^{(0)}_{\mathcal{T}_j}\cdot\ell^{(i)}}\geq
C\sum_{j=2}^{\infty}\beta_{(i)}^{-cj \delta^{(0)}\cdot\ell^{(i)}}=\infty,
\end{equation}
where to deduce the first inequality we again appeal to Rayleigh's monotonicity principle, and for the third we use the fact that there exists a deterministic constant $c\in (0,\infty)$, such that $\mathbb{P}$-a.s.,
\begin{equation}\label{llnx0regen}
j^{-1}X^{(0)}_{\mathcal{T}_j}\rightarrow c\delta^{(0)}.
\end{equation}
This is a consequence of the law of large numbers for $X^{(0)}$ and the law of large numbers for regeneration times that is stated at Lemma \ref{regenlem}(c). In conjunction with \cite[Theorem 2.11]{Barlow}, \eqref{rinfty} yields that $\mathcal{Z}^{(i)}$ is a $\bs{p}^{(i)}$-recurrent graph.

In the case when $\delta^{(0)}\cdot\ell^{(i)}=0$, a similar, but slightly more delicate argument can be applied to deduce the same conclusion. In particular, the bound given in the penultimate sum of \eqref{rinfty} is still applicable, and we need to show this is infinite. To do this, first observe Lemma \ref{regenlem} implies
\[\left(\left(X^{(0)}_{\mathcal{T}_{j+1}}-X^{(0)}_{\mathcal{T}_j}\right)\cdot\ell^{(i)}\right)_{j\geq 1}\]
form an i.i.d.\ sequence, and moreover, \eqref{llnx0regen} (together with \cite[Theorems VII.9.1 and VII.9.4]{Feller}) yields that these random variables have expectation equal to zero.  Thus the sum of these increments is a (non-trivial) one-dimensional simple random walk with zero mean, and standard arguments show that, $\mathbb{P}$-a.s.,
\[-\infty=\liminf_{j\rightarrow\infty} X^{(0)}_{\mathcal{T}_j}\cdot\ell^{(i)}<\limsup_{j\rightarrow\infty} X^{(0)}_{\mathcal{T}_j}\cdot\ell^{(i)}=\infty.\]
Hence, $\mathbb{P}$-a.s., there are infinitely many $j$ for which
\[\beta_{(i)}^{-X^{(0)}_{\mathcal{T}_j}\cdot\ell^{(i)}}\geq 1.\]
Thus $R^{(i)}(0,\infty)=\infty$ in this case as well.
\end{proof}

\begin{REM}\label{counterrem}
The condition $\delta^{(0)}\cdot \ell^{(i)}>0$ for transience (in each direction $\ell$ for which $\ell\cdot\delta^{(0)}>0$) permits some counterintuitive behaviour whereby a random walk can be transient in the ``wrong direction'' (to be more precise,  the walk $X^{(i)}$ can be transient in direction $\delta^{(0)}$ even though $\delta^{(0)}\cdot \delta^{(i)}<0$, and vice versa, See Example \ref{exa:counterint}).  Similar phenomena can be observed for example in random walk in i.i.d.~random environment, where in one dimension the transience of the walk is determined by the sign of $\E[\log(\omega_0^{-1}-1)]$ rather than $\E[\omega_0]$, where $\omega_0\in (0,1)$ is the random probability of stepping to the right at the origin \cite{Solomon}.
\end{REM}

\begin{EXA}
\label{exa:counterint}
Let
\[p^{(0)}(e_1)=p^{(0)}(e_2)=\frac{1}{4}+\vep, \quad p^{(0)}(-e_1)=p^{(0)}(-e_2)=\frac{1}{4}-\vep,\]
for some $\varepsilon\in (0,\frac14)$, so that $\delta^{(0)}=(2\varepsilon,2\varepsilon)$.   Furthermore, let
\begin{equation}
\label{funny1}
p^{(1)}(e_1)=\frac{15}{25},\qquad p^{(1)}(-e_1)=\frac{5}{25},\qquad p^{(1)}(e_2)=\frac{1}{25},\qquad p^{(1)}(-e_2)=\frac{4}{25}.
\end{equation}
In this case  $\delta^{(1)}=(\frac{10}{25},-\frac{3}{25})$ and $\hat{\ell}^{(1)}=(\log(3),-\log(4))$. Thus
\begin{equation}
\delta^{(1)}\cdot \delta^{(0)} > 0 >\ell^{(1)}\cdot\delta^{(0)},\label{signs1}
\end{equation}
which implies that even though a random walk on $\mathbb{Z}^d$ with transition probabilities $p^{(1)}$ is transient in the direction $\delta^{(0)}$, the process $X^{(1)}$ on $\mc{Z}^{(1)}$ is not.

If instead we rotate the transition probabilities \eqref{funny1}  through $\pi$ radians, i.e.\ we set
\[p^{(1)}(e_1)=\frac{5}{25},\qquad p^{(1)}(-e_1)=\frac{15}{25},\qquad p^{(1)}(e_2)=\frac{4}{25},\qquad p^{(1)}(-e_2)=\frac{1}{25},\]
then the signs in \eqref{signs1} are reversed, so $X^{(1)}$ is transient in the direction $\delta^{(0)}$, even though a walk on $\mathbb{Z}^d$ with transition probabilities $p^{(1)}$ is transient in the direction $-\delta^{(0)}$.
\end{EXA}

\section{Velocity}
\label{sec:velocity}

The main goal of this section is to prove Theorem \ref{thm:main1}(b).  We will frequently make use of regenerations of walks in direction $e_1$. In particular, we will write $\mc{T}=(\mc{T}_k)_{k \in \N}$ for the set of regeneration times for $X^{(0)}$, where, in the notation of the previous section, $\mc{T}_k=\mc{T}_k^{e_1}$. We note that, by definition, any regeneration time for a walk in direction $e_1$ is also a first hitting time of some level in direction $e_1$ by that walk. The corresponding regeneration levels are defined as $\{X^{(0)}_{\mc{T}_k}\cdot e_1: k \in \N\}$, and the corresponding regeneration points are $\{X^{(0)}_{\mc{T}_k}: k \in \N\}$.

\begin{LEM}
\label{lem:funny} Assume Condition \ref{cond:1}, and let $i\in\mathbb{N}\cup\{\infty\}$.
\begin{itemize}
\item[(1)] If $X^{(0)}$ is not transient in direction $\pm e_j$, then $X^{(i)}\cdot e_j$ almost surely hits 0 infinitely often, and $\lim_{n \to \infty}n^{-1}X_n^{(i)}\cdot e_j=0$.
\item[(2)] If for some fixed $\vep_1>0$, $\lim_{n \to \infty} n^{-1}X_n^{(i)}\cdot e_1=\vep_1$, $\P$-a.s., then $\lim_{n \to \infty} n^{-1}X_n^{(i)}=
(\vep_1/\delta^{(0)}_1)\delta^{(0)}$, $\P$-a.s.
\item[(3)] If $\lim_{n \to \infty} n^{-1}X_n^{(i)}\cdot e_1=0$, $\P$-a.s., then $\lim_{n \to \infty}n^{-1}X_n^{(i)}=0$.
\end{itemize}
\end{LEM}
\begin{proof} For (1), note that since the projection of $X^{(0)}$ in direction $e_j$ is a (lazy) simple symmetric random walk, it is recurrent.  This implies that $X^{(0)}$ has infinitely many $e_1$-direction regenerations when $X^{(0)}\cdot e_j>0$ and infinitely many when $X^{(0)}\cdot e_j<0$.  All of these regenerations are at points that all subsequent walks must pass through, which proves the first claim of (1).  For the second part of (1), note that a.s. there exists some finite random $Y$ such that $\sup_n n^{-2/3}X^{(0)}_n\cdot e_j\le Y$.  On the other hand there exists some $\vep>0$ and a random $N_0$ such that $X_n^{(0)}\cdot e_1>\vep n$ for all $n\ge N_0$ almost surely.  This means that, almost surely, for all $m$ sufficiently large, after walk $i$ has made $m>N_0$ steps, its first coordinate can be at most $m$, and therefore its second coordinate $X_m^{(i)}\cdot e_j$ can be at most $\sup_{k\le m/\vep}X_k^{(0)}\cdot e_j\le Y(m/\vep)^{2/3}$.  Therefore $X_m^{(i)}/m \ra 0$ almost surely.

To prove (2), note that by assumption $n^{-1}X_n^{(i)}\cdot e_1=\vep_1>0$ almost surely.  Let $M_n=\inf\{m: X_m^{(0)}=X^{(i)}_n\}$.  Then $M_n \ra \infty$ since $X^{(i)}$ is transient, and
\begin{align*}
\frac{X_n^{(i)}}{n}\cdot e_1=\frac{X_{M_n}^{(0)}}{M_n}\cdot e_1 \frac{M_n}{n}.
\end{align*}
The left hand side converges a.s.~to $\vep_1$ and $\frac{X_{M_n}^{(0)}}{M_n}\cdot e_1\to \delta^{(0)}_1$ a.s., so $\frac{M_n}{n} \to \vep_1/\delta^{(0)}_1$ a.s.  Thus,
\[\frac{X_n^{(i)}}{n}=\frac{X_{M_n}^{(0)}}{M_n}\frac{M_n}{n}\to \frac{\vep_1}{\delta^{(0)}_1}\delta^{(0)}.\]

It remains to prove (3).  Since $\delta^{(0)}\cdot e_1>0$ there exists $\delta>0$ and a random $N_1$ such that $X_n^{(0)}\cdot e_1>\delta n$ for every $n>N_1$.  Let $\vep\in (0,1)$.  Then there exists $N_2$ such that $|X_n^{(i)}\cdot e_1|< \vep\delta n$ for every $n\ge N_2$.   Let $N=N_2\vee (N_1/\vep)$.  Then for every $n>N\ge N_2$, $|X^{(0)}_{M_n}\cdot e_1|=|X_n^{(i)}\cdot e_1|\le \vep \delta n$.  This implies that $M_n\le \vep n$ so  $|X_n^{(i)}\cdot e_j|=|X^{(0)}_{M_n}\cdot e_j|\le M_n\le \vep n$. This establishes that $\limsup_{n \to \infty}n^{-1}|X_n^{(i)}\cdot e_j|\le \vep$, $\mathbb{P}$-a.s. for every $\vep\in (0,1)$, and therefore completes the proof.
\end{proof}

Next, let $\mc{L}^{(i)}=(L_1^{(i)},L_2^{(i)},\dots)$ denote the (ordered) set of strictly positive levels that are regeneration levels for every one of the walks $X^{(0)},\dots, X^{(i)}$.  The elements of $\mc{L}^{(i)}$ are called $i$-uber-regeneration levels. The following regeneration result is an easy consequence of the definition of the model. To state it, we use the notation $T^{(i)}_n=\inf\{k:X^{(i)}_k\cdot e_1=n\}$ to denote the first hitting time of level $n$ in the $e_1$ direction by walk $i$. Let $\hat{\mc{L}}^{(i)}=\{x\in \Z^d:x=X^{(i)}_{T_L^{(i)}} \text{ for some }L \in \mc{L}^{(i)}\}$ denote the corresponding uber-regeneration points.

\begin{LEM}
\label{lem:iid}
Fix $i \in \Z_+$.  If $L\in \mc{L}^{(i)}$, then $((X^{(r)}_{n+T_L^{(r)}}-X^{(r)}_{T_L^{(r)}})_{n\ge 0})_{r\le i}$:
\begin{itemize}
\item  is independent of
$((X^{(r)}_n)_{n\le T^{(r)}_{L}})_{r\le i}$, and
\item has the same law as $((X^{(r)}_{n+T_1^{(r)}}-X^{(r)}_{T_1^{(r)}})_{n\ge 0})_{r\le i}$ conditional on $\{1\in \mc{L}^{(i)}\}$.
\end{itemize}
\end{LEM}

Therefore, to find a common regeneration structure for the first $i$ walks, it suffices to show that $\P(\mc{L}^{(i)} \text{ is infinite})=1$. To this end, we start by checking that uber-regeneration levels occur at any natural number with strictly positive probability.

\begin{LEM}
\label{lem:uberprob}
 For each $i\in\mathbb{Z}_+$, there exists an $\varepsilon_i>0$ such that $\mathbb{P}(k\in\mc{L}^{(i)})=\varepsilon_i$ for every $k\geq 1$.
\end{LEM}
\begin{proof} Firstly, the directional transience of Theorem \ref{thm:main1}(a) and the strong Markov property imply the result in the case $i=0$. So, now suppose $i,k\geq 1$. We then have that
\[\mathbb{P}\left(k\in\mc{L}^{(i)}\right)=\mathbb{P}\left(k\in\mc{L}^{(0)}\right)\prod_{j=1}^i\mathbb{P}\left(k\in\mc{L}^{(j)}\:\vline\: k\in\mc{L}^{(j-1)}\right).\]
Conditioning on the graph $\mathcal{Z}^{(j)}$ (and noting that, on the event $k\in\mc{L}^{(j-1)}$, $X^{(0)}_{T_k^{(0)}}$ is a measurable function of $\mathcal{Z}^{(j)}$), we have that
\begin{eqnarray*}
\mathbb{P}\left(k\in\mc{L}^{(j)}\:\vline\: k\in\mc{L}^{(j-1)},\:\mathcal{Z}^{(j)}\right)&=&
P^{(j)}_{X^{(0)}_{T_k^{(0)}}}\left(X^{(j)}\mbox{ does not return to }X^{(0)}_{T_k^{(0)}}\right)\\
&=&\frac{1}{c^{(j)}\left(X^{(0)}_{T_k^{(0)}}\right)R^{(j)}\left(X^{(0)}_{T_k^{(0)}},\infty\right)}\\
&\geq &\frac{1}{\bar{c}^{(j)}\left(X^{(0)}_{T_k^{(0)}}\right)\sum_{m= T_k^{(0)}}^{\infty} r^{(j)}\left(X^{(0)}_m,X^{(0)}_{m+1}\right)},
\end{eqnarray*}
where $\bar{c}^{(i)}(x):=\sum_{y \in \Z^d:y\sim x}c^{(i)}(x,y)$. Note that we have applied \cite[Theorem 2.11]{Barlow} for the second inequality, and bounded the resistance to infinity as in the proof of Theorem \ref{thm:main1}(a). We thus obtain that
\begin{eqnarray*}
\lefteqn{\mathbb{P}\left(k\in\mc{L}^{(j)}\:\vline\: k\in\mc{L}^{(j-1)}\right)}\\
&\geq&
\mathbb{E}\left[\frac{1}{\bar{c}^{(j)}\left(X^{(0)}_{T_k^{(0)}}\right)\sum_{m= T_k^{(0)}}^{\infty} r^{(j)}\left(X^{(0)}_m,X^{(0)}_{m+1}\right)}\:\vline\:k\in\mc{L}^{(j-1)}\right]\\
&=&\mathbb{E}\left[\frac{1}{\bar{c}^{(j)}\left(X^{(0)}_{T_1^{(0)}}\right)\sum_{m= T_1^{(0)}}^{\infty} r^{(j)}\left(X^{(0)}_m,X^{(0)}_{m+1}\right)}\:\vline\:1\in\mc{L}^{(j-1)}\right],
\end{eqnarray*}
where the equality follows from Lemma \ref{lem:iid}. Now, as was observed in the proof of Theorem \ref{thm:main1}(a), the random variable $\sum_{m=0}^{\infty} r^{(j)}(X^{(0)}_m,X^{(0)}_{m+1})$ is $\mathbb{P}$-a.s.\ finite. Hence the above expectation is strictly positive. In particular, this confirms that
\begin{equation}\label{deltadef}
\delta_j:=\mathbb{P}\left(k\in\mc{L}^{(j)}\:\vline\: k\in\mc{L}^{(j-1)}\right)>0.
\end{equation}
Setting $\delta_0:=\varepsilon_0>0$, we thus have established the result with $\varepsilon_i=\prod_{j=0}^i\delta_j$.
\end{proof}

The above result allows us to establish the following lemma, which in view of Lemma \ref{lem:iid} confirms that there are indeed almost-surely an infinite number of $i$-uber-regeneration levels. In the proof, we write
\[\tau^{(r,i)}_{j}=\inf\left\{n:X^{(r)}_n=L^{(i)}_j\right\}\]
to represent the first hitting time of $L_j^{(i)} \in \mc{L}^{(i)}$ by walk $X^{(r)}$.

\begin{LEM}
\label{lem:L1}
For each $i\in \Z_+$,
\begin{itemize}
\item[(0)] $\P(L_1^{(i)}=1)>0$, and
\item[(1)] $L_1^{(i)}$ is almost surely finite, and
\item[(2)] given that $L_1^{(i)}=1$, $L_2^{(i)}$ is almost surely finite.
\end{itemize}
\end{LEM}
\begin{proof}[Sketch proof]
The claim (0) is immediate by Lemma \ref{lem:uberprob} with $k=1$.

The remaining two claims are trivial for $i=0$, and by Lemma  \ref{lem:iid} this implies that $\mc{L}^{(0)}$ is infinite a.s.
Proceeding by induction on $i$, armed with Lemma \ref{lem:iid}, we may prove the last two claims assuming that $\mc{L}^{(i-1)}$ is almost surely infinite.
Since $\mc{L}^{(i-1)}$ is an infinite set, at each time $\tau^{(i,i-1)}_{j}$, the walk $X^{(i)}$ has probability $\delta_i=\vep_i/\vep_{i-1}>0$ of never backtracking from its current location, and if this happens at some level $L\in \mc{L}^{(i-1)}$ then in fact  $L\in \mc{L}^{(i)}$.  If it fails then the walk $X^{(i)}$ reached some maximal level $M$ before backtracking, and we ask the question again when the walk reaches $M'=\inf\{n>M: n \in  \mc{L}^{(i-1)}\}$.  We can repeat this and eventually succeed to find an $i$-uber-regeneration level.  This shows that $\mc{L}^{(i)}$ is non-empty.  Once we find an $L_1^{(i)}\in \mc{L}^{(i)}$, when the walker $X^{(i)}$ reaches level $\inf\{n>L_1^{(i)}: n \in  \mc{L}^{(i-1)}\}$ it has probability at least $\delta_i>0$ of never backtracking from this location (not necessarily equal to $\delta_i$ because it does not backtrack past level $L_1^{(i)}$), and we can continue similarly to above to find $L_2^{(i)}\in \mc{L}^{(i)}$.
\end{proof}

Some care is required to make the above rigorous, and this is presented in the proof below. For an interval $\mc{I}\subset \Z$ write $\mc{Z}^{(i)}_{\mc{I}}:=\mc{Z}^{(i)}\cap (\mc{I}\times \Z^{d-1})$, and let $\mc{E}_{\mc{I}}^{(i)}$ denote the set of edges of $\mc{Z}^{(i)}_{\mc{I}}$ such that at least one of the end vertices $x$ has $x\cdot e_1$ in the interior of $\mc{I}$.  We start by giving a bound on a conditional version of the regeneration probability.  For this, first note that if $G$ is a connected graph with at least one edge then the graph can be identified with its set of edges.  Given a finite connected graph $H\subset \Z^d$ with at least two edges, a unique right-most vertex $u$, and a unique left-most vertex $o=(0,\dots,0)$, we let $E_H=\{\mc{E}^{(i)}_{[0,u\cdot e_1]}=H\}$ and
\[q_{H}=\P\left(u\in \hat{\mc{L}}^{(i)}\,\Big|\,E_H, u\in \hat{\mc{L}}^{(i-1)},1 \in \mc{L}^{(i)}\right).\]
\begin{LEM}
For every $H$, $q_H\ge \delta_i$, where $\delta_i$ was defined at \eqref{deltadef}.
\end{LEM}
\begin{proof}
Define $R=\{u \in \hat{\mc{L}}^{(i)}\}$, $R'=\{u\in \hat{\mc{L}}^{(i-1)}\}$, $F=\{X^{(i)}_n\cdot e_1\ge 1, \, \forall n\in [1,\tau]\}$ and
$F'=\{X^{(i)}_n\cdot e_1\ge 1, \,\forall n>\tau\}$, where $\tau$ is the hitting time of level $u$ by $X^{(i)}$.
Then,
\begin{align*}
q_H&=\P\left(R \Big|  E_H,R',F,F'    \right)\\
&=\dfrac{\P\left(R,E_H,R',F,F'    \right)}{\P\left(E_H,R',F,F'    \right)}\\
&=\dfrac{\P\left(R,E_H,R',F \right)}{\P\left(F', E_H,R',F    \right)}\\
&=\dfrac{\P(R,E_H,F\,|\,R')}{\P(F',E_H,F\,|\,R')}.
\end{align*}
Given $R'$, the event $R$ only depends on the environment $\mc{E}^{(i)}_{[u\cdot e_1-1,\infty)}=\mc{Z}^{(i)}_{[u\cdot e_1,\infty)}\cup \{(u-e_1,u)\}$ and the behaviour of the walker $X^{(i)}$ on this graph from the first hitting time of $u$ onwards.  Given $R'$, $E_H$  only depends on $\mc{E}^{(i)}_{(-\infty,u\cdot e_1]}$, and $F$ only depends on the behaviour of the walker $X^{(i)}$ on that graph until the first hitting time of $u$. Hence $R$ is conditionally independent of $E_H,F$ given $R'$ and
\begin{align*}
\dfrac{\P(R,E_H,F\,|\,R')}{\P(F',E_H,F\,|\,R')}=\dfrac{\P(R\,|\,R')\P(E_H,F\,|\,R')}{\P(F',E_H,F\,|\,R')}\ge \P(R\,|\,R'),
\end{align*}
The result follows since $\P(R\,|\,R')=\delta_i$.
\end{proof}

\begin{proof}[Proof of Lemma \ref{lem:L1}]
As indicated above, we prove (1) and (2) of the Lemma for $i \in \Z_+$ by induction on $i$.  For $i=0$ the result holds trivially.

For $i\in \N$, suppose that (1) and (2) hold for $i-1$.  In particular $\mc{L}^{(i-1)}$ is infinite almost surely.  Let $A=\{1\in \mc{L}^{(i-1)}\}$ and $B=\{1\in \mc{L}^{(i)}\}$.  Then $\P(B)=\vep_i>0$ and $\P(B|A)=\delta_i>0$.

To prove part (1) of the lemma we construct a finite-time random walk $\tilde{X}^{(i)}$ on a finite piece of graph $\tilde{\mc{Z}}^{(i)}$, based on the following collection of mutually independent random variables:
\begin{itemize}
\item Let $(Z_{\{0\}}, X_{\{0\}})$ have the joint law of the almost-surely finite graph with edge set $\mc{E}^{(i)}_{(-\infty,L_1^{(i-1)}]}$ and a $p^{(i)}$-random walk on it, started from $0$, run until time $\tau_1^{(i,i-1)}$.  {\em Below is a pictorial example of $(Z_{\{0\}}, X_{\{0\}})$.  The graph $Z_{\{0\}}$ is black, and the edges traversed by the walker $X^{(i)}$ up to the first hitting time of $L_1^{(i-1)}=5$ are the (pink) shadow.}
\begin{center}
\begin{tikzpicture}[scale=.5]
\node[circle,fill=black,scale=.2,label=below :{$0$}] (AA1) at (0,0) {};
\node[circle,fill=black,scale=.2,label=below:{$$}] (AA2) at (-1,0) {};
\node[circle,fill=black,scale=.2,label=below:{$$}] (AA3) at (-1,-1) {};
\node[circle,fill=black,scale=.2,label=below:{$$}] (AA4) at (-1,-2) {};
\node[circle,fill=black,scale=.2,label=below:{$$}] (AA5) at (-1,-3) {};
\node[circle,fill=black,scale=.2,label=below:{$$}] (AA6) at (-2,-3) {};
\node[circle,fill=black,scale=.2,label=below:{$$}] (AA7) at (-2,-2) {};
\node[circle,fill=black,scale=.2,label=below:{$$}] (AA8) at (-2,-1) {};
\node[circle,fill=black,scale=.2,label=below:{$$}] (AA9) at (-3,-2) {};

\node[circle,fill=black,scale=.2,label=below:{$$}] (AB1) at (1,0) {};
\node[circle,fill=black,scale=.2,label=below:{$$}] (AB2) at (2,0) {};
\node[circle,fill=black,scale=.2,label=below:{$$}] (AB3) at (2,-1) {};
\node[circle,fill=black,scale=.2,label=below:{$$}] (AB4) at (2,-2) {};
\node[circle,fill=black,scale=.2,label=below:{$$}] (AB5) at (3,-2) {};
\node[circle,fill=black,scale=.2,label=below:{$$}] (AB6) at (3,-1) {};
\node[circle,fill=black,scale=.2,label=below:{$$}] (AB7) at (3,0) {};
\node[circle,fill=black,scale=.2,label=below:{$$}] (AB8) at (4,0) {};
\node[circle,fill=black,scale=.2,label=below:{$$}] (AB9) at (4,-1) {};
\node[circle,fill=black,scale=.2,label=below:{$$}] (AB10) at (4,-2) {};
\node[circle,fill=black,scale=.4,label=right:{$$}] (AB11) at (5,-2) {};

\node[circle,fill=black,scale=.2,label=below:{$$}] (AC1) at (1,1) {};
\node[circle,fill=black,scale=.2,label=below:{$$}] (AC2) at (2,1) {};
\node[circle,fill=black,scale=.2,label=below:{$$}] (AC3) at (3,1) {};

\draw[color=pink,line width=2] (AA1)--(AA2)--(AA3)--(AA2)--(AA1)--(AB1)--(AB2)--(AB1)--(AC1)--(AC2)--(AC3)--(AB7)--(AB8)--(AB9)--(AB10)--(AB11);
\draw (AA1)--(AA2)--(AA3)--(AA5)--(AA6)--(AA7)--(AA8)--(AA3);
\draw (AA4)--(AA9);
\draw (AA1)--(AB1)--(AB2)--(AB3)--(AB4)--(AB5)--(AB6)--(AB7)--(AB8)--(AB9)--(AB10)--(AB11);
\draw (AB1)--(AC1)--(AC2)--(AC3)--(AB7);
\node (D1) at (5,2) {};
\node[label=below:$L_1^{(i-1)}$] (D2) at (5,-4) {};
\draw[dotted] (D1)--(D2);
\end{tikzpicture}
\end{center}
\item Conditional on $A\cap B^c$, let $M\ge 1$ denote the largest level reached by $X^{(i)}$ before it returns to level 0 (at time $T_0$), and we let $\bar{M}=\inf\{n>M: n \in  \mc{L}^{(i-1)}\}$.    Let $(Z_{\{k\}}, X_{\{k\}})_{k\in \N}$ be i.i.d.~with the joint law of
    \[\left(\mc{E}^{(i)}_{[0,\bar{M}]}, (X^{(i)}_m)_{m\in [T^{(i)}_1, T_0]}\right),\]
conditional on $A\cap B^c$.  Note that $X^{(i)}_{T^{(i)}_1}=X^{(i-1)}_{T^{(i-1)}_1}$ and $X^{(i)}_{T_0}=X^{(i-1)}_{T^{(i-1)}_1}-e_1$ since $1\in \mc{L}^{(i-1)}$. In other words, conditional on $A\cap B^c$ we look at $\mc{Z}^{(i)}$, and run the random walk $X^{(i)}$ from the hitting time of level 1, until it hits level 0.  The walk $X^{(i)}$ has not looked at $\mc{Z}^{(i)}_{[\bar{M},\infty)}$, so we will only remember that part up to level $\bar{M}$.  {\em Below is a pictorial example of $(Z_{\{1\}}, X_{\{1\}})$.  The graph $Z_{\{1\}}$ is black, and the edges traversed by the walker $X^{(i)}$ up to the first hitting time of $L_1^{(i-1)}=5$ are the (pink) shadow.  Here $M=12$ and $\bar{M}=L_4^{(i-1)}=13$.}
\begin{center}
\begin{tikzpicture}[scale=.5]
\node[circle,fill=black,scale=.2,label=below :{$$}] (AA1) at (0,0) {};

\node[circle,fill=black,scale=.2,label=below:{$$}] (AB1) at (1,0) {};
\node[circle,fill=black,scale=.2,label=below:{$$}] (AB2) at (2,0) {};
\node[circle,fill=black,scale=.2,label=below:{$$}] (AB3) at (2,-1) {};
\node[circle,fill=black,scale=.2,label=below:{$$}] (AB4) at (2,-2) {};
\node[circle,fill=black,scale=.2,label=below:{$$}] (AB5) at (3,-2) {};
\node[circle,fill=black,scale=.2,label=below:{$$}] (AB6) at (3,-1) {};
\node[circle,fill=black,scale=.2,label=below:{$$}] (AB7) at (3,0) {};
\node[circle,fill=black,scale=.2,label=below:{$$}] (AB8) at (4,0) {};
\node[circle,fill=black,scale=.2,label=below:{$$}] (AB9) at (4,-1) {};
\node[circle,fill=black,scale=.2,label=below:{$$}] (AB10) at (4,-2) {};
\node[circle,fill=black,scale=.4,label=right:{$$}] (AB11) at (5,-2) {};

\node[circle,fill=black,scale=.2,label=below:{$$}] (AC1) at (1,1) {};
\node[circle,fill=black,scale=.2,label=below:{$$}] (AC2) at (2,1) {};

\node[circle,fill=black,scale=.2,label=below:{$$}] (AC3) at (6,-2) {};
\node[circle,fill=black,scale=.2,label=below:{$$}] (AC4) at (7,-2) {};
\node[circle,fill=black,scale=.2,label=below:{$$}] (AC5) at (8,-2) {};
\node[circle,fill=black,scale=.2,label=below:{$$}] (AC6) at (9,-2) {};
\node[circle,fill=black,scale=.2,label=below:{$$}] (AC7) at (9,-3) {};
\node[circle,fill=black,scale=.2,label=below:{$$}] (AC8) at (9,-4) {};
\node[circle,fill=black,scale=.2,label=below:{$$}] (AC9) at (8,-4) {};
\node[circle,fill=black,scale=.2,label=below:{$$}] (AC10) at (7,-4) {};
\node[circle,fill=black,scale=.2,label=below:{$$}] (AC11) at (10,-4) {};
\node[circle,fill=black,scale=.2,label=below:{$$}] (AC12) at (10,-3) {};
\node[circle,fill=black,scale=.4,label=below:{$$}] (AC13) at (11,-4) {};
\node[circle,fill=black,scale=.2,label=below:{$$}] (AC14) at (12,-4) {};
\node[circle,fill=black,scale=.2,label=below:{$$}] (AC15) at (13,-4) {};
\node[circle,fill=black,scale=.2,label=below:{$$}] (AC16) at (13,-3) {};
\node[circle,fill=black,scale=.2,label=below:{$$}] (AC17) at (13,-2) {};
\node[circle,fill=black,scale=.2,label=below:{$$}] (AC18) at (13,-1) {};
\node[circle,fill=black,scale=.4,label=below:{$$}] (AC19) at (14,-1) {};

\node[circle,fill=black,scale=.2,label=below:{$$}] (AD1) at (5,-1) {};
\node[circle,fill=black,scale=.2,label=below:{$$}] (AD2) at (6,-1) {};

\node[circle,fill=black,scale=.2,label=below:{$$}] (AE1) at (13,0) {};
\node[circle,fill=black,scale=.2,label=below:{$$}] (AE2) at (12,0) {};
\node[circle,fill=black,scale=.2,label=below:{$$}] (AE3) at (12,-1) {};

\node[circle,fill=black,scale=.2,label=below:{$$}] (AF1) at (10,-2) {};
\node[circle,fill=black,scale=.2,label=below:{$$}] (AF2) at (10,-1) {};
\node[circle,fill=black,scale=.2,label=below:{$$}] (AF3) at (9,-1) {};

\draw[color=pink,line width=2] (AA1)--(AB1)--(AB2)--(AB3)--(AB4)--(AB5)--(AB6)--(AB7)--(AB8)--(AB9)--(AB10)--(AB11)--(AD1)--(AD2)--(AC3)--(AC4)--(AC5)--(AC6)--(AF3)--(AC6)--(AC7)--(AC8)--(AC11)--(AC12)--(AC11)--(AC13)--(AC14)--(AC15)--(AC16);
\draw[color=pink,line width=2] (AF3)--(AF2)--(AF1)--(AC12);
\draw (AA1)--(AB1)--(AB2)--(AB3)--(AB4)--(AB5)--(AB6)--(AB7)--(AB8)--(AB9)--(AB10)--(AB11)--(AC3)--(AC4)--(AC5)--(AC6)--(AC7)--(AC8)--(AC9)--(AC10);
\draw (AC8)--(AC11)--(AC12)--(AC11)--(AC13)--(AC14)--(AC15)--(AC16)--(AC17)--(AC18)--(AC19);
\draw (AB1)--(AC1)--(AC2);
\draw (AB3)--(AB6);
\draw (AB11)--(AD1)--(AD2)--(AC3);
\draw (AC18)--(AE1)--(AE2)--(AE3)--(AC18);
\draw (AC11)--(AF1)--(AF2)--(AF3)--(AC6);

\node (D1) at (5,2) {};
\node[label=below:$L_2^{(i-1)}$] (D2) at (5,-6) {};
\draw[dotted] (D1)--(D2);
\node (D3) at (11,2) {};
\node[label=below:$L_3^{(i-1)}$] (D4) at (11,-6) {};
\draw[dotted] (D3)--(D4) {};
\node (D5) at (14,2) {};
\node[label=below:$L_4^{(i-1)}$] (D6) at (14,-6) {};
\draw[dotted] (D5)--(D6) {};

\node (D7) at (1,2) {};
\node[label=below:$1$] (D8) at (1,-6) {};
\draw[dotted] (D7)--(D8) {};
\end{tikzpicture}
\end{center}

\item For finite (connected) graphs $H\subset \Z^d$ with at least two edges, a marked point $o$, and whose (unique) right-most vertex is a leaf, $v$,  we let $(X^{H}_{\{k\}})_{k \in \mathbb{N}}$ be i.i.d. with law that of a $p^{(i)}$ walk on $H$, started at $o$, observed until the hitting time of $v$.
\item Let $(W_{\{k\}})_{k\in \N}$ be i.i.d. $U[0,1]$ random variables.
\end{itemize}

We will also need the notion of graph addition (concatenation).  Suppose that $H$ and $H'$ are finite connected subgraphs of $\Z^d$ such that $H$ has a unique right most vertex $v$, and $H'$ has a unique left most vertex $v'$.  We write $H+H'\subset \Z^d$ to denote the graph obtained by ``concatenation with a shared edge'', identifying $v-e_1$ with $v'$ and identifying $v$ with $v'+e_1$, as in the example below.
\begin{center}
\begin{tikzpicture}[scale=.5]
\node[circle,fill=black,scale=.2,label=below :{$v$}] (AA1) at (0,0) {};
\node[circle,fill=black,scale=.2,label=below:{$$}] (AA2) at (-1,0) {};
\node[circle,fill=black,scale=.2,label=below:{$$}] (AA3) at (-1,-1) {};
\node[circle,fill=black,scale=.2,label=below:{$$}] (AA4) at (-1,-2) {};
\node[circle,fill=black,scale=.2,label=below:{$$}] (AA5) at (-1,-3) {};
\node[circle,fill=black,scale=.2,label=below:{$$}] (AA6) at (-2,-3) {};
\node[circle,fill=black,scale=.2,label=below:{$$}] (AA7) at (-2,-2) {};
\node[circle,fill=black,scale=.2,label=below:{$$}] (AA8) at (-2,-1) {};
\node[circle,fill=black,scale=.2,label=below:{$$}] (AA9) at (-3,-2) {};
\node (AAA) at (-1,-4) {$H$};
\draw (AA1)--(AA2)--(AA3)--(AA5)--(AA6)--(AA7)--(AA8)--(AA3);
\draw (AA4)--(AA9);
\end{tikzpicture}
\hspace{.5cm}
\begin{tikzpicture}[scale=.5]
\node[circle,fill=black,scale=.2,label=below :{$v'$}] (AA1) at (0,0) {};
\node[circle,fill=black,scale=.2,label=below:{$$}] (AB1) at (1,0) {};
\node[circle,fill=black,scale=.2,label=below:{$$}] (AB2) at (2,0) {};
\node[circle,fill=black,scale=.2,label=below:{$$}] (AB3) at (2,-1) {};
\node[circle,fill=black,scale=.2,label=below:{$$}] (AB4) at (2,-2) {};
\node[circle,fill=black,scale=.2,label=below:{$$}] (AB5) at (3,-2) {};
\node[circle,fill=black,scale=.2,label=below:{$$}] (AB6) at (3,-1) {};
\node[circle,fill=black,scale=.2,label=below:{$$}] (AB7) at (3,0) {};
\node[circle,fill=black,scale=.2,label=below:{$$}] (AB8) at (4,0) {};
\node[circle,fill=black,scale=.2,label=below:{$$}] (AB9) at (4,-1) {};
\node[circle,fill=black,scale=.2,label=below:{$$}] (AB10) at (4,-2) {};
\node[circle,fill=black,scale=.2,label=right:{$$}] (AB11) at (5,-2) {};
\node[circle,fill=black,scale=.2,label=below:{$$}] (AC1) at (1,1) {};
\node[circle,fill=black,scale=.2,label=below:{$$}] (AC2) at (2,1) {};
\node[circle,fill=black,scale=.2,label=below:{$$}] (AC3) at (3,1) {};
\node (AAA) at (2,-3) {$H'$};
\draw (AA1)--(AB1)--(AB2)--(AB3)--(AB4)--(AB5)--(AB6)--(AB7)--(AB8)--(AB9)--(AB10)--(AB11);
\draw (AB1)--(AC1)--(AC2)--(AC3)--(AB7);
\end{tikzpicture}
\hspace{1cm}
\begin{tikzpicture}[scale=.5]
\node[circle,fill=black,scale=.2,label=below :{$v$}] (AA1) at (0,0) {};
\node[circle,fill=black,scale=.2,label=below:{$$}] (AA2) at (-1,0) {};
\node[circle,fill=black,scale=.2,label=below:{$$}] (AA3) at (-1,-1) {};
\node[circle,fill=black,scale=.2,label=below:{$$}] (AA4) at (-1,-2) {};
\node[circle,fill=black,scale=.2,label=below:{$$}] (AA5) at (-1,-3) {};
\node[circle,fill=black,scale=.2,label=below:{$$}] (AA6) at (-2,-3) {};
\node[circle,fill=black,scale=.2,label=below:{$$}] (AA7) at (-2,-2) {};
\node[circle,fill=black,scale=.2,label=below:{$$}] (AA8) at (-2,-1) {};
\node[circle,fill=black,scale=.2,label=below:{$$}] (AA9) at (-3,-2) {};
\draw (AA1)--(AA2)--(AA3)--(AA5)--(AA6)--(AA7)--(AA8)--(AA3);
\draw (AA4)--(AA9);
\node[circle,fill=black,scale=.2,label=above :{$v'$}] (AA1') at (-1,0) {};
\node[circle,fill=black,scale=.2,label=below:{$$}] (AB1) at (0,0) {};
\node[circle,fill=black,scale=.2,label=below:{$$}] (AB2) at (1,0) {};
\node[circle,fill=black,scale=.2,label=below:{$$}] (AB3) at (1,-1) {};
\node[circle,fill=black,scale=.2,label=below:{$$}] (AB4) at (1,-2) {};
\node[circle,fill=black,scale=.2,label=below:{$$}] (AB5) at (2,-2) {};
\node[circle,fill=black,scale=.2,label=below:{$$}] (AB6) at (2,-1) {};
\node[circle,fill=black,scale=.2,label=below:{$$}] (AB7) at (2,0) {};
\node[circle,fill=black,scale=.2,label=below:{$$}] (AB8) at (3,0) {};
\node[circle,fill=black,scale=.2,label=below:{$$}] (AB9) at (3,-1) {};
\node[circle,fill=black,scale=.2,label=below:{$$}] (AB10) at (3,-2) {};
\node[circle,fill=black,scale=.2,label=right:{$$}] (AB11) at (4,-2) {};
\node[circle,fill=black,scale=.2,label=below:{$$}] (AC1) at (0,1) {};
\node[circle,fill=black,scale=.2,label=below:{$$}] (AC2) at (1,1) {};
\node[circle,fill=black,scale=.2,label=below:{$$}] (AC3) at (2,1) {};
\node (AAA) at (0,-4) {$H+H'$};
\draw (AA1')--(AB1)--(AB2)--(AB3)--(AB4)--(AB5)--(AB6)--(AB7)--(AB8)--(AB9)--(AB10)--(AB11);
\draw (AB1)--(AC1)--(AC2)--(AC3)--(AB7);
\end{tikzpicture}
\end{center}

Firstly, set the first block of $\tilde{\mc{Z}}^{(i)}$ equal to $Z_{\{0\}}$, and note that the (unique) right most vertex $u_0$ of this finite graph is a leaf.  Let $\tilde{X}^{(i)}$ follow $X_{\{0\}}$ (in other words, the increments of $\tilde{X}^{(i)}$ are those of $X_{\{0\}}$ until those increments are exhausted).

If $W_{1}<\delta_i$ then we set $\tilde{L}^{(i)}_1=u_0\cdot e_1$.  If $W_{1}\ge \delta_i$ then we let the next block of $\tilde{\mc{Z}}^{(i)}$ be $Z_{\{1\}}$ and we let $\tilde{X}^{(i)}$ follow $X_{\{1\}}$.  Let $H_1$ be the finite graph $Z_{\{0\}}+Z_{\{1\}}$, with marked point $u_0-e_1$.  The (unique) right-most point of this graph is $u_1$.  We then let $\tilde{X}^{(i)}$ follow $X^{H_1}_{\{1\}}$ (it is exhausted when it reaches $u_1$).  If $W_{2}<\delta_i$ then we set $\tilde{L}^{(i)}_1$ to be $u_1\cdot e_1$ (the right most level of $H_1$).   Otherwise we continue recursively.  Suppose that $W_{k}\ge \delta_i$ for each $k\le k_1$.  Then we let the next block of $\tilde{\mc{Z}}^{(i)}$ be $Z_{\{k_1\}}$, and we let $\tilde{X}^{(i)}$ follow $X_{\{k_1\}}$.  Let $H_{k_1}$ be the finite graph $\sum_{k=0}^{k_1}Z_{\{k\}}$, with marked point $u_{k_1-1}-e_1$, where $u_{k_1-1}$ is the rightmost point of  $\sum_{k\le k_1-1}Z_{\{k\}}$.  Let $\tilde{X}^{(i)}$ follow $X^{H_{k_1}}_{\{1\}}$ (it is exhausted when it reaches $u_{k_1}$).  Let $K=\inf\{j:W_{k}<\delta_i\}$ and we set $\tilde{L}^{(i)}_1=u_{K}\cdot e_1$.
Set  $\tilde{\mc{Z}}=\sum_{k\le K}Z_{\{k\}}$.

By construction the first claim of the lemma holds for $\tilde{X}^{(i)}$.  We claim that the pair $(\tilde{X}^{(i)},\tilde{\mc{Z}})$ has the same law as $(({X}^{(i)}_n)_{n\leq L_1^{(i)}},\mc{E}_{(-\infty,L^{(i)}_1]})$.  To see this, note that it is trivially true up to the first hitting time of $L_1^{(i-1)}$.  Whether $L_1^{(i-1)}$ is a regeneration level for $X^{(i)}$ is independent of  $\mc{E}^{(i)}_{(-\infty,L_1^{(i-1)}]}$ and therefore the probability of this event is $\vep'_i$.  Conditional on this event not occurring, we can examine the joint law of the environment $\mc{E}^{(i)}_{[L_1^{(i-1)}-1,\infty)}$ and the walk $X^{(i)}$ on it (starting from $X^{(i-1)}_{\tau^{(i-1,i-1)}_1}$).  In particular the walker $X^{(i)}$ only walks a finite distance $M$ (possibly 0) in direction $e_1$ before backtracking to $L_1^{(i-1)}-1$, and the joint law of the part of the environment up to the first level $\bar{M}$ in $\mc{L}^{(i-1)}$ that is greater than $M$, together with $X^{(i)}$ on this part of the graph is that of $(Z_{\{1\}},X_{\{1\}})$.  The environment to the right of $\bar{M}$ is independent of that to the left, so we can delete it and generate a new independent copy without changing the distribution. The walker $\tilde{X}^{(i)}$ then has the law of $X^{(i)}$ (conditional on the environment up to level $\bar{M}$) until it hits level $\bar{M}$.  Whether or not $\bar{M}$ is a regeneration level for $X^{(i)}$ is independent of $\mc{E}^{(i)}_{(-\infty,\bar{M}]}$, and thus we continue to have the correct law up until the first time that the walks reach a level in $\mc{L}^{(i)}$.

We now prove the second claim similarly, by constructing a finite-time random walk $\tilde{X}'^{(i)}$ on a finite piece of graph $\tilde{\mc{Z}}'^{(i)}$ (which need only be constructed from level $1$ onwards), based on the following collection of mutually independent random variables:

\begin{itemize}
\item Let $(Z'_{\{0\}}, X'_{\{0\}})$ have the joint law of \[
\left(\mc{E}^{(i)}_{(-\infty,L_2^{(i-1)}]}, (X^{(i)}_m)_{m\le \tau^{(i,i-1)}_{2}}\right)\] conditional on $B$.  In other words, conditional on $B$ we look at that part of $\mc{Z}^{(i)}$ up to the next $(i-1)$-uber-regeneration level, and run the random walk $X^{(i)}$ on it until it hits this level.  Given $B$, after first reaching level 1 it must hit this level without returning to level  $0$. {\em In the pictorial example below, the walker first backtracks from 0, but upon reaching level $1=L^{(i)}_1$ it can never backtrack past level 1, and we observe the walk until it hits $L^{(i-1)}_2$.}
\begin{center}
\begin{tikzpicture}[scale=.5]
\node[circle,fill=black,scale=.2,label=below :{$0$}] (AA1) at (0,0) {};
\node[circle,fill=black,scale=.2,label=below:{$$}] (AA2) at (-1,0) {};
\node[circle,fill=black,scale=.2,label=below:{$$}] (AA3) at (-1,-1) {};
\node[circle,fill=black,scale=.2,label=below:{$$}] (AA4) at (-1,-2) {};
\node[circle,fill=black,scale=.2,label=below:{$$}] (AA5) at (-1,-3) {};
\node[circle,fill=black,scale=.2,label=below:{$$}] (AA6) at (-2,-3) {};
\node[circle,fill=black,scale=.2,label=below:{$$}] (AA7) at (-2,-2) {};
\node[circle,fill=black,scale=.2,label=below:{$$}] (AA8) at (-2,-1) {};
\node[circle,fill=black,scale=.2,label=below:{$$}] (AA9) at (-3,-2) {};

\node[circle,fill=black,scale=.2,label=below:{$$}] (AB1) at (1,0) {};
\node[circle,fill=black,scale=.2,label=below:{$$}] (AB2) at (2,0) {};
\node[circle,fill=black,scale=.2,label=below:{$$}] (AB3) at (2,-1) {};
\node[circle,fill=black,scale=.2,label=below:{$$}] (AB4) at (2,-2) {};
\node[circle,fill=black,scale=.2,label=below:{$$}] (AB5) at (3,-2) {};
\node[circle,fill=black,scale=.2,label=below:{$$}] (AB6) at (3,-1) {};
\node[circle,fill=black,scale=.2,label=below:{$$}] (AB7) at (3,0) {};
\node[circle,fill=black,scale=.2,label=below:{$$}] (AB8) at (4,0) {};
\node[circle,fill=black,scale=.2,label=below:{$$}] (AB9) at (4,-1) {};
\node[circle,fill=black,scale=.2,label=below:{$$}] (AB10) at (4,-2) {};
\node[circle,fill=black,scale=.4,label=right:{$$}] (AB11) at (5,-2) {};

\node[circle,fill=black,scale=.2,label=below:{$$}] (AC1) at (1,1) {};
\node[circle,fill=black,scale=.2,label=below:{$$}] (AC2) at (2,1) {};
\node[circle,fill=black,scale=.2,label=below:{$$}] (AC3) at (3,1) {};

\draw[color=pink,line width=2] (AA1)--(AA2)--(AA3)--(AA2)--(AA1)--(AB1)--(AB2)--(AB1)--(AC1)--(AC2)--(AC3)--(AB7)--(AB8)--(AB9)--(AB10)--(AB11);
\draw (AA1)--(AA2)--(AA3)--(AA5)--(AA6)--(AA7)--(AA8)--(AA3);
\draw (AA4)--(AA9);
\draw (AA1)--(AB1)--(AB2)--(AB3)--(AB4)--(AB5)--(AB6)--(AB7)--(AB8)--(AB9)--(AB10)--(AB11);
\draw (AB1)--(AC1)--(AC2)--(AC3)--(AB7);
\node (D1) at (5,2) {};
\node[label=below:$L_2^{(i-1)}$] (D2) at (5,-4) {};
\draw[dotted] (D1)--(D2);

\node (D3) at (1,2) {};
\node[label=below:$L_1^{(i)}$] (D4) at (1,-4) {};
\draw[dotted,line width=2] (D3)--(D4);
\end{tikzpicture}
\end{center}
\item Let $(Z'_{\{k\}}, X'_{\{k\}})_{k\in \N}\sim (Z_{\{k\}}, X_{\{k\}})_{k\in \N}$, where the latter appeared in the previous construction.
\item For finite (connected) graphs $H\subset \Z^d$ with at least 2 edges, a marked point $o$, and whose (unique) right-most vertex is (a leaf) $v$, and whose unique left-most vertex is (a leaf) $u$ we let $(X^{'H}_{\{k\}})_{k \in \mathbb{N}}$ be i.i.d. with law that of a $p^{(i)}$ walk on $H$, started at $o$ conditional on hitting $v$ before $u$, observed until the hitting time of $v$.
\item Let $(W'_{\{k\}})_{k\in \N}$ be i.i.d. $U[0,1]$ random variables.
\end{itemize}
Set the first block of $\tilde{\mc{Z}}'^{(i)}$ equal to $Z'_{\{0\}}$ and we let $\tilde{X}'^{(i)}$ follow $X'_{\{0\}}$ (it is exhausted when it reaches the right-most point $v_0$ of $Z'_{\{0\}}$). Set $H'_{\{1\}}=\mc{Z}'_{\{0\}}$.

If $W'_{1}<q_{H'_{\{1\}}}$ then set $\tilde{L}'^{(i)}_{2}=v_0\cdot e_1$.  Otherwise $W'_{1}\ge q_{H'_{\{1\}}}$ and we proceed to define subwalks on subgraphs as follows.  Suppose that $W'_{k}\ge q_{H'_{\{k\}}}$ for $k\in [k'_1]$, and let $v_k$ be the unique right-most vertex of $H'_{\{k\}}$.    Then we set the next block of $\tilde{\mc{Z}}'^{(i)}$ equal to $Z'_{\{k'_1\}}$ and let $\tilde{X}'^{(i)}$ follow $X'_{\{k'_1\}}$.  Let $H'_{k'_1+1}=H'_{\{k'_1\}}+Z'_{\{k'_1\}}$, with marked point $v_{k'_1}-e_1$.  Then we let $\tilde{X}'^{(i)}$ follow ${X'}^{H'_{k'_1+1}}_{\{k'_1\}}$.

Since $q_{H'_{\{k\}}}>\delta_i$ for each $k$, we have that $K':=\inf\{k:W'_{k}\ge q_{H'_{\{k\}}}\}$ is almost surely finite, and we set $\tilde{L}'^{(i)}_2=v_{K'}\cdot e_1$.  Set $\tilde{\mc{Z}}'=H'_{K'}$.  That $(\tilde{\mc{Z}}',\tilde{X}'^{(i)})$ has the desired properties can now be checked as above.
\end{proof}

The final ingredient we need to deduce the existence of velocities is the finiteness of the expectation of the inter-regeneration distance.

\begin{LEM}
For each $i \in \Z_+$,
\[\E[L^{(i)}_2-L_1^{(i)}]=\vep_i^{-1}<\infty.\]
\end{LEM}
\begin{proof}
This closely follows the proof of \cite[Lemma 3.2.5]{Zeitouni} (attributed to Martin Zerner), but note that we deal with the (common) regeneration levels rather than the regeneration times (since these are not common to all walks).

By Lemma \ref{lem:uberprob}, $\vep_i=\lim_{\ell\to \infty}\P(\ell \in \mc{L}^{(i)})$. On the other hand,
\begin{align*}
\lim_{\ell\to \infty}\P(\ell \in \mc{L}^{(i)})&=\lim_{\ell\to \infty}\P(\exists k\ge 2: \ell =L^{(i)}_k)\\
&=\lim_{\ell\to \infty}\sum_{n \ge 1}\P(L^{(i)}_1=n,\exists k\ge 2: L^{(i)}_k-L^{(i)}_1=\ell-n)\\
&=\lim_{\ell\to \infty}\sum_{n \ge 1}\P(L^{(i)}_1=n)\P(\exists k\ge 2: L^{(i)}_k-L^{(i)}_1=\ell-n).
\end{align*}
By the renewal theorem and the fact that $\P(L^{(i)}_2-L^{(i)}_1=m)>0$ for each $m\in \N$ we have
\[
\lim_{\ell\to \infty}\P(\exists k\ge 2: L^{(i)}_k-L^{(i)}_1=\ell-n)=\E[L^{(i)}_2-L_1^{(i)}]^{-1}.\]
It is then easy to show that
\[\lim_{\ell\to \infty}\sum_{n \ge 1}\P(L^{(i)}_1=n)\P(\exists k\ge 2: L^{(i)}_k-L^{(i)}_1=\ell-n)=\E[L^{(i)}_2-L_1^{(i)}]^{-1},\]
and the result follows.
\end{proof}

\begin{proof}[{\bf Proof of Theorem \ref{thm:main1}(b)}]
It now follows from standard arguments invoking the law of large numbers for the i.i.d.~sequence of processes between inter-regeneration levels that, almost surely,
\[\frac{X^{(i)}_n}{n}\cdot e_1\to \frac{\E[L^{(i)}_{2}-L^{(i)}_{1}]}{\E[\tau^{(i,i)}_{2}-\tau^{(i,i)}_{1}]},\]
 where the right hand side is deterministic, and is equal to zero if and only if $\E[\tau^{(i,i)}_{2}-\tau^{(i,i)}_{1}]=\infty$. If $v^{(i)}\cdot e_1>0$ then Lemma \ref{lem:funny}(2) applies, and otherwise Lemma \ref{lem:funny}(3) applies to give the result.
\end{proof}

\section{Ballistic phase}\label{balsec}

The aim of this section is to prove part (a) of Theorem \ref{thm:main2}, which characterises the ballistic phase. This will be done by estimating the expected time for the $i$th random walk $X^{(i)}$ to progress a unit distance in the $\ell^{(i)}$ direction (see Lemma 4.2). First, though, we confirm that $t_{(i)}$, as introduced below \eqref{varphidef}, is well-defined and can be used to control the backtracking probability for the original random walk $X^{(0)}$ in the direction $\ell^{(i)}$. We recall the notation $\varphi_{(i)}(t)=\mathbb{E}[\exp\{-tX^{(0)}_1\cdot \ell^{(i)}\}]$ from \eqref{varphidef}.

\begin{LEM}\label{backtrackprob} Assume Condition \ref{cond:1}. For each $i\in\mathbb{N}\cup\{\infty\}$, there exists a unique solution $t_{(i)}>0$ to the equation $\varphi_{(i)}(t)=1$. Moreover, for all $h\geq 0$,
\[\mathbb{P}\left(-\min_{n\geq 0}X_n^{(0)}\cdot\ell^{(i)}\geq h\right)\leq \ee^{-t_{(i)}h}.\]
\end{LEM}
\begin{proof} Fix $i\in\mathbb{N}\cup\{\infty\}$. Clearly $\varphi_{(i)}(0)=1$.  Since $\ell^{(i)}\neq 0$, we have $\ell^{(i)}\cdot e>0$ for some $e\in\{\pm e_j:\:j=1,\dots,d\}$.  Moreover,
\[\varphi_{(i)}(t)\ge p^{(0)}(-e)\exp\{t e\cdot \ell^{(i)}\},\]
so $\varphi_{(i)}(t)\rightarrow\infty$ as $t\rightarrow\infty$. Next, as the sum of convex functions, at least one of which is strictly convex, we have that $\varphi_{(i)}$ is a strictly convex function. Thus to check the first claim of the lemma it will suffice to show that $\varphi_{(i)}'(0)<0$. To this end, we compute the relevant derivative directly,
\[\varphi_{(i)}'(0)=-\mathbb{E}[X^{(0)}_1\cdot \ell^{(i)}]=-\delta^{(0)}\cdot\ell^{(i)},\]
and observe that this is strictly negative by Condition \ref{cond:1}(c). Given the existence of $t_{(i)}>0$, the remaining claim of the lemma is an immediate application of Lundberg's inequality \cite[Corollary II.3.4]{AA}.
\end{proof}

To state the next lemma, we define the stopping times
\[T^{(i,j)}_{n}:=\inf\left\{m:\:X^{(i)}_m\cdot \ell^{(j)}\geq n\right\},\]
which under Condition \ref{cond:1} are almost surely finite by Theorem \ref{thm:main1}(a),
and recall that $\alpha_{(i)}:=\ee^{t_{(i)}}>1$.

\begin{LEM}\label{fmom} Assume Condition \ref{cond:1}. For each $i\in\mathbb{N}\cup\{\infty\}$, if $\beta_{(i)}<\alpha_{(i)}$, then
\[\sup_{n\geq 0}\mathbb{E}\left[T^{(i,i)}_{n+1}-T^{(i,i)}_{n}\right]<\infty.\]
\end{LEM}
\begin{proof} For $n\in \mathbb{R}$, define $\mathcal{Z}^{(i)}_n:=\{x\in\mathcal{Z}^{(i)}: x\cdot \ell^{(i)}\in[n,n+1)\}$. We then have that
\[\mathbb{E}\left[T^{(i,i)}_{n+1}-T^{(i,i)}_{n}\right]\leq \mathbb{E}\left[\max_{x\in\mathcal{Z}^{(i)}_n}E^{(i)}_x\left[T^{(i,i)}_{n+1}\right]\right],\]
where $E_x^{(i)}$ denotes expectation with respect to the law $P_x^{(i)}$ of a $\bs{p}^{(i)}$ random walk on $\mc{Z}^{(i)}$ starting from $x\in \mc{Z}^{(i)}$.
Now, setting
\[T^{(i,j)}_{-m,n}:=\inf\left\{k:\:X^{(i)}_k\cdot \ell^{(j)}\not\in(-m,n)\right\},\]
the monotone convergence theorem yields that, for $\mathbb{P}$-a.e.\ realisation of $\mathcal{Z}^{(i)}$, and every $x$ in the finite set $\mathcal{Z}^{(i)}_n$, we have that
\[E^{(i)}_x\left[T^{(i,i)}_{n+1}\right]=\lim_{m\rightarrow\infty}E^{(i)}_x\left[T^{(i,i)}_{-m,n+1}\right].\]
Moreover, applying well-known estimates for random walks on electrical networks (e.g.\ characterisations of the Green's function in terms of effective resistance from \cite[Theorem 1.31 and Proposition 2.55]{Barlow}), we have that
\[E^{(i)}_x\left[T^{(i,i)}_{-m,n+1}\right]\leq R^{(i)}(x,\mathcal{Z}^{(i)}_{n+1})\sum_{l=-m+1}^{n}\sum_{y\in\mathcal{Z}^{(i)}_l}c^{(i)}(y),\]
where $R^{(i)}(x,\mathcal{Z}^{(i)}_{n+1})$ is the effective resistance from $x$ to $\mathcal{Z}^{(i)}_{n+1}$ in $\mathcal{Z}^{(i)}$ (equipped with the conductances $c^{(i)}(\cdot,\cdot)$). Hence (with $\bar{c}$ as in the proof of Lemma \ref{lem:uberprob}) we obtain that
\begin{align}
\lefteqn{\mathbb{E}\left[T^{(i,i)}_{n+1}-T^{(i,i)}_{n}\right]}\nonumber\\
&\leq
\mathbb{E}\left[\max_{x\in\mathcal{Z}^{(i)}_n}R^{(i)}(x,\mathcal{Z}^{(i)}_{n+1})\sum_{m\leq n}\sum_{y\in\mathcal{Z}^{(i)}_m}c^{(i)}(y)\right]\nonumber\\
&\leq \mathbb{E}\left[\sum_{l=T^{(0,i)}_{n}}^{\bar{T}^{(0,i)}_{n}}r^{(i)}(X^{(0)}_l,X^{(0)}_{l+1})\sum_{m\leq \bar{T}^{(0,i)}_{n}}\bar{c}^{(i)}(X^{(0)}_m)\mathbf{1}_{\{X^{(0)}_m\cdot\ell^{(i)}<n+1\}}\right],\label{expectation}
\end{align}
where $\bar{T}^{(0,i)}_{n}$ is the time of the last visit by $X^{(0)}$ to the set $\mathcal{Z}^{(i)}_{X^{(0)}_{T_n^{(0,i)}}\cdot\ell^{(i)}}$. We will estimate this expectation by decomposing into various pieces. Firstly, define
\[a_{(i),n}:=\beta_{(i)}^{X_{T^{(0,i)}_{n}}^{(0)}\cdot\ell^{(i)}},\]
and then set
\[\Upsilon:=a_{(i),n}\sum_{m={T}^{(0,i)}_{n}}^{\bar{T}^{(0,i)}_{n}}r^{(i)}(X^{(0)}_m,X^{(0)}_{m+1}),\]
\[\Gamma_1:=a_{(i),n}^{-1}\sum_{m\leq T^{(0,i)}_{n}-1}\bar{c}^{(i)}(X^{(0)}_m),\]
and
\[\Gamma_2:=a_{(i),n}^{-1}\sum_{m=T^{(0,i)}_{n}}^{\bar{T}^{(0,i)}_{n}}\bar{c}^{(i)}(X^{(0)}_m)\mathbf{1}_{\{X^{(0)}_m\cdot\ell^{(i)}<n+1\}}.\]
Note that $\Upsilon$ only depends on the path $(X^{(0)}_{T^{(0,i)}_{n}+m}-X^{(0)}_{T^{(0,i)}_{n}})_{m\geq 0}$, and so is independent of $\Gamma_1$.  Letting $\|X\|_p:=\E[|X|^p]^{1/p}$, we see that for any $p,q>1$ such that $p^{-1}+q^{-1}=1$, the expectation at \eqref{expectation} is bounded above by
\[\mathbb{E}\left[\Upsilon\left(\Gamma_1+\Gamma_2\right)\right]=\mathbb{E}\left[\Upsilon\Gamma_1\right] +\mathbb{E}\left[\Upsilon\Gamma_2\right]\le
\mathbb{E}\left[\Upsilon\right]\mathbb{E}\left[\Gamma_1\right]+
\left\|\Upsilon\right\|_p
\left\|\Gamma_2\right\|_q.\]
To bound the terms involving $\Upsilon$, observe that, for some $C>1$,
\begin{eqnarray*}
\lefteqn{\sup_{n\geq 0}\mathbb{P}\left(\max_{m\in\{T^{(0,i)}_{n},\dots,\bar{T}^{(0,i)}_{n}\}}a_{(i),n}r^{(i)}(X^{(0)}_m,X^{(0)}_{m+1})\geq \lambda\right)}\\
&\leq&
\sup_{n\geq 0}\mathbb{P}\left(\max_{m\in\{T^{(0,i)}_{n},\dots,\bar{T}^{(0,i)}_{n}\}}C\beta_{(i)}^{\left(X_{T^{(0,i)}_{n}}^{(0)}-X_m^{(0)}\right)\cdot\ell^{(i)}}\geq \lambda\right).
\end{eqnarray*}
Taking logs and using the Markov property and Lemma \ref{backtrackprob}, for all $\lambda>C$ this is at most
\begin{eqnarray*}
\mathbb{P}\left(-\min_{m\geq 0}X_m^{(0)}\cdot\ell^{(i)}\geq \frac{\log(\lambda/C)}{\log \beta_{(i)}}\right)\le  \ee^{-\frac{t_{(i)}\log(\lambda/C)}{\log \beta_{(i)}}}\le c\lambda^{-\frac{\log\alpha_{(i)}}{\log \beta_{(i)}}}.
\end{eqnarray*}
In particular, for $p$ sufficiently close to 1,
\[\sup_{n\geq 0}\left\|\max_{m\in\{T^{(0,i)}_{n},\dots,\bar{T}^{(0,i)}_{n}\}}a_{(i),n}r^{(i)}(X^{(0)}_m,X^{(0)}_{m+1})\right\|_{p}<\infty.\]
 Since
\[\left\|\Upsilon\right\|_p\leq \left\|\max_{m\in\{T^{(0,i)}_{n},\dots,\bar{T}^{(0,i)}_{n}\}}a_{(i),n}r^{(i)}(X^{(0)}_m,X^{(0)}_{m+1})\right\|_{p^2}\left\|\bar{T}_{n}^{(0,i)}+1-T_{n}^{(0,i)}\right\|_{pq},\]
to complete the proof that $\sup_{n\geq 0}\|\Upsilon\|_p<\infty$, it will thus be sufficient to check that
\begin{equation}\label{fff}
\sup_{n\geq 0}\left\|\bar{T}_{n}^{(0,i)}-T_{n}^{(0,i)}\right\|_q<\infty
\end{equation}
for arbitrary $q\geq 1$, as we will do below. We continue for the moment, however, by dealing with $\Gamma_1$, which satisfies
\begin{eqnarray*}
\Gamma_1&\leq&C\sum_{m\leq T^{(0,i)}_{n}-1}\beta_{(i)}^{\left(X^{(0)}_m-X_{T^{(0,i)}_{n}}^{(0)}\right)\cdot\ell^{(i)}}\\
&\leq &C\sum_{m\leq n}\beta_{(i)}^{m+1-n}\#\left\{l\geq 0:\:X^{(0)}_l\in\mathcal{Z}_m^{(i)}\right\}.
\end{eqnarray*}
Thus, taking expectations, we find that
\begin{equation}\label{ggg}
\sup_{n\geq 0}\E[\Gamma_1]\leq C\sup_{n\in\mathbb{Z}}\E\left[\#\left\{l\geq 0:\:X^{(0)}_l\in\mathcal{Z}_n^{(i)}\right\}\right]\leq \frac{C\sup_{n\geq 0}\E\left[\tilde{T}^{(0,i)}_{n}-T^{(0,i)}_{n}\right]}{\mathbb{P}\left(\min_{m\geq 0}X_m^{(0)}\cdot\ell^{(i)}\geq 0\right)},
\end{equation}
where $\tilde{T}^{(0,i)}_{n}$ is the first time $m$ after $T_n^{(0,i)}$ such that $(X_m^{(0)}-X_{T_n^{(0,i)}}^{(0)})\cdot\ell^{(i)}\geq 1$, and we have obtained the expression in the right-hand side by considering the number of returns by $X^{(0)}$ from $\mathcal{Z}_{n+1}^{(i)}$ to $\mathcal{Z}_{n}^{(i)}$, which is stochastically dominated by a geometric
$\mathbb{P}(\min_{m\geq 0}X_m^{(0)}\cdot\ell^{(i)}\geq 0)$ random variable, and noting that the duration of each excursion back into $\mathcal{Z}_n^{(i)}$ is stochastically dominated by $\tilde{T}^{(0,i)}_{n}-T^{(0,i)}_{n}$. Now, \cite[Theorem II]{Doney} implies that
\begin{equation}\label{ttail}
\sup_{n\geq 0}\mathbb{P}\left(\tilde{T}_{n}^{(0,i)}-T_{n}^{(0,i)}\geq \lambda\right)\leq C\ee^{-c\lambda},
\end{equation}
which in turn yields the numerator on the right-hand side of (\ref{ggg}) is finite. Moreover, we also have that $\mathbb{P}(\min_{m\geq 0}X_m^{(0)}\cdot\ell^{(i)}\geq 0)>0$, and so we have established the desired bound for $\Gamma_1$. Finally, the definition of $\Gamma_2$ implies
\[\Gamma_2\leq C\sum_{m=T^{(0,i)}_{n}}^{\bar{T}^{(0,i)}_{n}}\beta_{(i)}^{\left(X^{(0)}_m-X_{T^{(0,i)}_{n}}^{(0)}\right)\cdot\ell^{(i)}}\mathbf{1}_{\{X^{(0)}_m\cdot\ell^{(i)}<n+1\}}\leq C\left(\bar{T}^{(0,i)}_{n}-T^{(0,i)}_{n}\right).\]
Now, arguing as in the previous paragraph, we have that $\bar{T}^{(0,i)}_{n}-T^{(0,i)}_{n}$ is stochastically dominated by $\sum_{i=1}^G\xi_i$, where $G$ is a geometric
$\mathbb{P}(\min_{m\geq 0}X_m^{(0)}\cdot\ell^{(i)}\geq 0)$ random variable, and $(\xi_i)_{i\geq 1}$ are independent copies of $\tilde{T}^{(0,i)}_{n}-T^{(0,i)}_{n}$, independent of $G$. We thus obtain, by a simple adaptation of the argument leading to Wald's identity, that, for integers $q\in\mathbb{N}$,
\[\sup_{n\geq 0}\mathbb{E}[\Gamma_2^q]\leq C \mathbb{E}[G^q]\mathbb{E}[\xi^q],\]
and hence we obtain from the finiteness of the moments of the geometric distribution and \eqref{ttail} that $\sup_{n\geq 0}\|\Gamma_2\|_q<\infty$ for arbitrary $q\geq 1$. Since the latter part of the argument also implies (\ref{fff}), the proof is complete.
\end{proof}

\begin{proof}[{\bf \em Proof of Theorem \ref{thm:main2}(a)}]
From Theorem \ref{thm:main1} and the definition of $T_n^{(i,i)}$, we have
\[\frac{X^{(i)}_{T_n^{(i,i)}}\cdot \ell^{(i)}}{T_n^{(i,i)}}\to v:=v^{(i)}\cdot \ell^{(i)},\]
and
\[\frac{X^{(i)}_{T_n^{(i,i)}}\cdot \ell^{(i)}}{n}\to 1\]
almost surely.  If $v=0$, then from the above we have that $n^{-1}T_n^{(i,i)} \to \infty$ almost surely, so for any $M$,
\[\P(n^{-1}T_n^{(i,i)}>M) \to 1.\]
as $n\rightarrow\infty$. However, by Lemma \ref{fmom} there exists $C>0$ such that $\E[T_n^{(i,i)}]\le Cn$ for every $n$.  Together with Markov's inequality, this gives
\[ \P( n^{-1}T_n^{(i,i)}>M)\le \frac{\E[n^{-1}T_n^{(i,i)}]}{M}\le \frac{C}{M}.\]
For $M>2C$ this is less than $1/2$, and hence we conclude that $v\ne 0$.
\end{proof}

\section{Sub-ballistic phase}\label{subsec}

In this section, we describe conditions under which the walk $X^{(i)}$ is sub-ballistic, namely we prove Theorem \ref{thm:main2}(b). Throughout we work under the assumption that Condition \ref{cond:1} holds. To establish the main result, we consider a regeneration structure for $X^{(0)}$ in the direction $\ell^{(i)}$ for each $i$, as defined above Lemma \ref{regenlem}. Moreover, we need to define what it means to have a trap of height $h$ for the walk $X^{(i)}$ in the $j$th regeneration block for $X^{(0)}$ with respect to the direction $\ell^{(i)}$. In particular, let $\mathcal{E}^{(i)}_j(h)$ be the event that there exist $m,n$ with $\mathcal{T}^{\ell^{(i)}}_j\leq m\leq n\leq \mathcal{T}^{\ell^{(i)}}_{j+1}$ such that $X_m^{(0)}$ and $X_n^{(0)}$ are cut-points for $X^{(0)}$ (i.e.\ for $n'\in\{m,n\}$ we have that $\{X_{m'}^{(0)}:\:m'\leq n'\}\cap\{X_{m'}^{(0)}:\:m'> n'\}=\emptyset$) and
\begin{equation}\label{inc1}
\left\lfloor \left(X_m^{(0)}-X_{n}^{(0)}\right)\cdot \ell^{(i)}\right\rfloor = h.
\end{equation}
The key ingredient of the argument to demonstrate sub-ballisticity is the following, which reveals that the probability of creating a trap in a regeneration block decays at an exponential rate no greater than that of the backtracking probability for $X^{(0)}\cdot\ell^{(i)}$ (recall Lemma \ref{backtrackprob}). (NB.\ Establishing the reverse inequality is much easier, but we do not need it in what follows.)

\begin{LEM}\label{trapprob} Assume Condition \ref{cond:1}. For each $i\in\mathbb{N}\cup\{\infty\}$,
\[\limsup_{h\rightarrow\infty}\frac{-\log\mathbb{P}\left(\mathcal{E}^{(i)}_j(h)\right)}{h}\leq t_{(i)},\]
where $t_{(i)}$ was defined below \eqref{varphidef}. (NB. By the fact regeneration blocks are identically distributed, the left-hand side is independent of $j\geq 1$.)
\end{LEM}

Before proving this result, we explain how it is used to establish Theorem \ref{thm:main2}(b). To this end, for $n\in \mathbb{N}$ and $\varepsilon\in(0,1)$, we introduce a parameter
\[h^{(i)}_{n,\varepsilon}=\frac{(1-\varepsilon)\log n}{t_{(i)}}\]
that will be used to define what it means for a trap to be big for $X^{(i)}$ at scale $n$. In particular, we set
\[N^{(i)}_{n,\varepsilon}:=\#\left\{j\in\{1,\dots,n-1\}\::\: \mathcal{E}^{(i)}_j(h^{(i)}_{n,\varepsilon})\mbox{ occurs}\right\}.\]
The following lemma tells us that this random variable grows polynomially with $n$.

\begin{LEM}\label{nlem} Assume Condition \ref{cond:1}, and let $i\in\mathbb{N}\cup\{\infty\}$. For any $\varepsilon\in(0,1)$, it holds that
\[\lim_{n\rightarrow\infty}\mathbb{P}\left(N_{n,\varepsilon}^{(i)}\geq n^{\varepsilon/2}\right)=1.\]
\end{LEM}
\begin{proof} Lemma \ref{trapprob} gives us that, for large $n$,
\[\mathbb{P}\left(\mathcal{E}^{(i)}_j(h^{(i)}_{n,\varepsilon})\right)\geq \alpha_{(i)}^{-\left(\frac{1-3\varepsilon/4}{1-\varepsilon}\right)h^{(i)}_{n,\varepsilon}}=n^{-1+3\varepsilon/4}.\]
Hence, applying the i.i.d.\ structure of regeneration blocks of Lemma \ref{regenlem}, and writing $\mathrm{Bin}(n,p)$ for a binomial random variable with parameters $n$ and $p$,
\begin{eqnarray*}
\mathbb{P}\left(N^{(i)}_{n,\varepsilon}<n^{\varepsilon/2}\right)&\leq& \mathbb{P}\left(\mathrm{Bin}\left(n-1,n^{-1+3\varepsilon/4}\right)< n^{\varepsilon/2}\right)\\
&\leq & \frac{\mathrm{Var}\left(\mathrm{Bin}\left(n-1,n^{-1+3\varepsilon/4}\right)\right)}{\left((n-1)n^{-1+3\varepsilon/4}-n^{\varepsilon/2}\right)^2}\\
&\leq &C n^{-3\varepsilon/4},
\end{eqnarray*}
which completes the proof.
\end{proof}

We next present a lemma which gives a lower bound for the probability of the time we spend in a trap being small. We introduce the notation $\sigma_j^{(i)}$ to represent the hitting time of $X^{(0)}_{\mathcal{T}^{\ell^{(i)}}_j}$ by $X^{(i)}$, and recall the notation $P^{(i)}_x$ from the proof of Lemma \ref{fmom}.

\begin{LEM}\label{tinc} Assume Condition \ref{cond:1}, and let $i\in\mathbb{N}\cup\{\infty\}$. There exists a deterministic constant $c$ such that, on the event $\mathcal{E}^{(i)}_j(h)$ with $h\geq 1$,
\[P^{(i)}_{X^{(0)}_{\mathcal{T}_j}}\left(\sigma_{j+1}^{(i)}\geq \beta_{(i)}^{h}\right)\geq c.\]
\end{LEM}
\begin{proof} Suppose $\mathcal{E}^{(i)}_j(h)$ holds, and let $x=X_m^{(0)}$ and $y=X_n^{(0)}$ be the two vertices appearing in the definition of this event. (These are not necessarily uniquely defined, but that is unimportant for the proof.) Let
\[p=p_{\mc{Z}^{(i)}}:=P^{(i)}_x\left(X^{(i)}\mbox{ hits $y$ before returning to $x$}\right),\]
and note that, with $u=X^{(0)}_{\mathcal{T}_j}$,
\[P^{(i)}_{u}\left(\sigma_{j+1}^{(i)}\geq \beta_{(i)}^{h}\right)\geq \mathbb{P}\left(\mathrm{Geo}\left(p\right)\geq \beta_{(i)}^{h}\right)=(1-p)^{\beta_{(i)}^{h}},\]
where (given $\mc{Z}^{(i)}$) $\mathrm{Geo}(p)$ is a geometric random variable taking values in $\N$ with parameter $p=p_{\mc{Z}^{(i)}}$. We thus are motivated to find an upper bound for $p$. For this, we observe from \cite[Exercise 2.62]{LP}, for example, that with $z:=X^{(0)}_{m+1}$,
\[p\leq  P^{(i)}_z\left(X^{(i)}_{0}\mbox{ hits $y$ before $x$}\right)=\frac{r^{(i)}\left(x,z\right)}{R^{(i)}\left(x,y\right)},\]
Now,
\[r^{(i)}\left(x,z\right)=c^{(i)}\left(x,z\right)^{-1}\leq C \beta_{(i)}^{-x\cdot\ell^{(i)}},\]
and
\[R^{(i)}\left(x,y\right)\geq R^{(i)}\left(y,\{w:\:w\sim_{(i)}y\}\right) \geq \frac{\min_{w:\:w\sim_{(i)}y}r^{(i)}\left(y,w\right)}{\#\{w:\:w\sim_{(i)}y\}}\geq C\beta_{(i)}^{-y\cdot\ell^{(i)}},\]
where we use the notation $w\sim_{(i)}y$ to mean that $w$ is a neighbour of $y$ in $\mathcal{Z}^{(i)}$. Combining these estimates with \eqref{inc1} thus yields
\[p\leq C \beta_{(i)}^{-(x-y)\cdot\ell^{(i)}}\leq C \beta_{(i)}^{-h}.\]
Hence we obtain that
\[P^{(i)}_{u}\left(\sigma_{j+1}^{(i)}\geq \beta_{(i)}^{h}\right)\geq \left(1-C \beta_{(i)}^{-h}\right)^{\beta_{(i)}^{h}},\]
which is clearly bounded away from $0$ as $h\rightarrow\infty$, and the result follows.
\end{proof}

We use the previous two lemmas to establish that the random walk $X^{(i)}$ takes a super-linear time to reach the $n$th regeneration level.

\begin{LEM}\label{tbig} Assume Condition \ref{cond:1}, and let $i\in\mathbb{N}\cup\{\infty\}$. If $\beta_{(i)}>\alpha_{(i)}$, then there exists a deterministic constant $\delta>0$ such that
\[\lim_{n\rightarrow\infty}\mathbb{P}\left(\sigma_{n}^{(i)}\geq n^{1+\delta}\right)=1.\]
\end{LEM}
\begin{proof} On the event that $N^{(i)}_{n,\varepsilon}\geq n^{\varepsilon/2}$, Lemma \ref{tinc} implies that $\sigma_{n}^{(i)}$ is stochastically bounded below by a random variable of the form $\beta_{(i)}^{h^{(i)}_{n,\varepsilon}}\mathrm{Bin}(n^{\varepsilon/2},c)$. Note that the multiplying constant here is given by
\[\beta_{(i)}^{h^{(i)}_{n,\varepsilon}}=n^{\frac{(1-\varepsilon)\log \beta_{(i)}}{t_{(i)}}},\]
which, by our assumption that $\beta_{(i)}>\alpha_{(i)}=\ee^{t_{(i)}}$, can be made greater than $n$ by choosing $\varepsilon>0$ suitably small. Hence for all such $\vep$ we have that
\begin{align}
\mathbb{P}\left(\sigma_{n}^{(i)}\geq n^{1+\delta}\right)&\geq \mathbb{P}\left(\sigma_{n}^{(i)}\geq n^{1+\delta}\big| N^{(i)}_{n,\varepsilon}\geq n^{\varepsilon/2} \right)\mathbb{P}\left(N^{(i)}_{n,\varepsilon}\geq n^{\varepsilon/2}\right)\nn\\
&\ge \mathbb{P}\left(\mathrm{Bin}(n^{\varepsilon/2},c)\geq n^{\delta}\right)\mathbb{P}\left(N^{(i)}_{n,\varepsilon}\geq n^{\varepsilon/2}\right).\label{banana1}
\end{align}
Taking $\delta<\varepsilon/2$ and applying Lemma \ref{nlem}, we see that \eqref{banana1} converges to $1$ as $n\rightarrow\infty$, which establishes the desired conclusion.
\end{proof}

From the preceding result, the proof of Theorem \ref{thm:main2}(b) is straightforward.

\begin{proof}[{\bf \em Proof of Theorem \ref{thm:main2}(b)}] By Lemma \ref{tbig}, we have that
\[\lim_{n\rightarrow\infty}\mathbb{P}\left(X_{n^{1+\delta}}^{(i)}\cdot \ell^{(i)}\leq X^{(0)}_{\mathcal{T}^{\ell^{(i)}}_n}\cdot  \ell^{(i)}\right)=1.\]
However, we also know from Lemma \ref{regenlem} that there exists finite constant $C$ such that
\[\lim_{n\rightarrow\infty}\mathbb{P}\left(X^{(0)}_{\mathcal{T}^{\ell^{(i)}}_n}\cdot  \ell^{(i)}\leq Cn\right)=1.\]
Hence we conclude that
\[\lim_{n\rightarrow\infty}\mathbb{P}\left(X_{n^{1+\delta}}^{(i)}\cdot  \ell^{(i)}\leq Cn\right)=1,\]
which yields
\[n^{-1}X_{n}^{(i)}\cdot  \ell^{(i)}\rightarrow 0\]
in probability. Since we also know the almost-sure limit of the left-hand side exists (by Theorem \ref{thm:main1}), we in fact have that the above convergence holds almost-surely. Moreover, recalling Condition \ref{cond:1}(c) and the fact that $v^{(i)}$ must be a scalar multiple of $\delta^{(0)}$ (i.e.\ Theorem \ref{thm:main1}(b)), we conclude that $v^{(i)}=0$.
\end{proof}

To complete the section, we need to prove Lemma \ref{trapprob}. Towards deducing an appropriate lower bound for the probability of $\mathcal{E}_j^{(i)}(h)$, we will present a collection of events that together ensure the existence of a trap, and for which we can suitably estimate the probability of their intersection. At the heart of the proof is the application of a large deviations result, the usefulness of which depends on the identification of the direction in which $X^{(0)}$ is most likely to drift, conditional on the projection $X^{(0)}\cdot\ell^{(i)}$ backtracking. In particular, the latter direction transpires to be described by
\[\hat{\delta}^{(i)}:=\mathbb{E}\left[\ee^{-t_{(i)}X_1^{(0)}\cdot\ell^{(i)}}X_1^{(0)}\right].\]
(Note that, from the properties of $\varphi_{(i)}$ set out in the proof of Lemma \ref{backtrackprob}, it holds that
\[\hat{\delta}^{(i)}\cdot\ell^{(i)}=-\varphi_{(i)}'(t_{(i)})<0,\]
and so if $X^{(0)}$ follows the direction $\hat{\delta}^{(i)}$, then $X^{(0)}\cdot\ell^{(i)}$ is indeed decreasing.)
\begin{REM}
Although we will not need this fact in the proof of Lemma \ref{trapprob}, we observe that $\hat{\delta}^{(i)}$ is the drift of the Markov process with transition probabilities given by the Doob transform
\[\hat{p}^{(0)}(y-x)=\frac{{p}^{(0)}(y-x)h^{(i)}(y)}{h^{(i)}(x)},\]
where $h^{(i)}$ is the $X^{(0)}$ harmonic function given by $h^{(i)}(x):=\ee^{-t_{(i)}x\cdot\ell^{(i)}}$. At least in cases where $X^{(0)}$ can only enter the region $\{x:\:x\cdot\ell^{(i)}\leq -h\}$ by hitting the hyperplane $L_h:=\{x:\:x\cdot\ell^{(i)}=-h\}$ (as is the case in Example \ref{exa:canonical}), the transition probabilities $\hat{p}^{(0)}$ precisely describe the law of $X^{(0)}$ conditioned to hit $L_h$ up to this hitting time (from which time the process proceeds with transition probability $p^{(0)}$).
\end{REM}
\begin{proof}[{\bf \em Proof of Lemma \ref{trapprob}}] Given $\varepsilon\in(0,1)$ and $h\geq 1$, we start by defining a sequence of eight events. In understanding their definition, it will aid the reader to refer to Figure \ref{trapfig}. First, let
\[x_a=X^{(0)}_{\mathcal{T}_j^{\ell^{(i)}}},\]
and consider three cylinders $C_1$, $C_2$, $C_3$, as shown in Figure \ref{trapfig}. In particular, $C_1$ has straight sides parallel to $\delta^{(0)}$, radius $\varepsilon h$, is positioned so that its centre line runs through $x_a$, and it has height chosen so that it just touches the planes
\[P_1:=\left\{x:\:x\cdot \ell^{(i)}=X^{(0)}_{\mathcal{T}_j^{\ell^{(i)}}}\cdot\ell^{(i)}\right\}\]
and
\[P_2:=\left\{x:\:x\cdot \ell^{(i)}=X^{(0)}_{\mathcal{T}_j^{\ell^{(i)}}}\cdot\ell^{(i)}+h\right\}.\]
We define $C_2$ similarly, but parallel to $\hat{\delta}^{(i)}$, and positioned so that the ends of $C_1$ and $C_2$ closest to the plane $P_2$ are separated by a distance of $\varepsilon h$. The cylinder $C_3$ is also defined in the same way, again parallel to ${\delta}^{(0)}$, and positioned so that the ends of $C_2$ and $C_3$ closest to the plane $P_3$ are separated by a distance of $\varepsilon h$. It is moreover possible to complete the above construction in such a way that the three cylinders are disjoint. In addition to these sets, we define points:
\begin{itemize}
  \item $x_b$, to be the lattice point closest to the point $y_{b,\vep}$, where $y_{b,\vep}$ is the point on the centreline of $C_1$ at distance $2\varepsilon h$ from the end nearest to $P_1$;
  \item $x_c$, to be a lattice point lying within distance $2\varepsilon h$ of both $C_1$ and $C_2$, but separated by at least one lattice step from one of these sets, and satisfying $(x_c-x_a)\cdot\ell^{(i)}\in[h(1+\varepsilon/4),h(1+3\varepsilon/4)]$;
  \item $x_d$, to be the lattice point closest to the point $y_{d,\vep}$, where $y_{d,\vep}$ is the point on the centreline of $C_2$ at distance $\varepsilon h$ from the end nearest to $P_2$;
  \item $x_e$, to be a lattice point lying within distance $2\varepsilon h$ of both $C_2$ and $C_3$, but separated by at least one lattice step from one of these sets, and satisfying $\lfloor (x_c-x_e)\cdot\ell^{(i)}\rfloor =h$;
  \item $x_f$, to be the lattice point closest to the point $y_{f,\vep}$, where $y_{f,\vep}$ is the point on the centreline of $C_3$ at distance $\varepsilon h$ from the end nearest to $P_1$;
  \item $x_g$, to be a lattice point that lies within distance $4\varepsilon h$ of $C_3$, and satisfies
  $x_g\cdot\ell^{(i)}\geq X^{(0)}_{\mathcal{T}_j^{\ell^{(i)}}}\cdot\ell^{(i)}+h(1+3\varepsilon)$.
\end{itemize}
We next define the events of interest $(E_k)_{k=1}^8$.
\begin{itemize}
  \item $E_1=\{$After ${\mathcal{T}_j^{\ell^{(i)}}}$, $X^{(0)}$ follows a shortest lattice path to $x_b$, follows the same path back to $x_a$, and then returns again along that path to $x_b\}$. (The recrossing of the path ensures that none of the points on it are regenerations.)
  \item $E_2=E_1\cap \{$From $x_b$, $X^{(0)}$ stays inside $C_1$ until it reaches a distance $\leq \varepsilon h$ from the end of $C_1$ closest to $P_2\}$.
  \item $E_3=E_2\cap \{$From the previous stopping time, $X^{(0)}$ follows a shortest simple lattice path that passes through $x_c$ and eventually reaches $x_d\}$.
  \item $E_4=E_3\cap \{$From $x_d$, $X^{(0)}$ stays inside $C_2$ until it reaches a distance $\leq \varepsilon h$ from the end of $C_2$ closest to $P_1\}$.
  \item $E_5=E_4\cap \{$From the previous stopping time, $X^{(0)}$ follows a shortest, simple lattice path that passes through $x_e$ and eventually reaches $x_f$, whilst satisfying $(X^{(0)}-x_a)\cdot\ell^{(i)}\geq 0\}$.
  \item $E_6=E_5\cap \{$From $x_f$, $X^{(0)}$ stays inside $C_3$ until it reaches a distance $\leq \varepsilon h$ from the end of $C_3$ closest to $P_2\}$.
  \item $E_7=E_6\cap \{$From the previous stopping time, $X^{(0)}$ follows a shortest lattice path to $x_g\}$.
  \item $E_8=E_7\cap \{$From $x_g$, $X^{(0)}\cdot \ell^{(i)}$ never drops below $x_g\cdot\ell^{(i)}\}$.
\end{itemize}
By construction, for any fixed $\varepsilon\in(0,1/4)$ and large enough $h$, we have that $\mathcal{E}_j^{(i)}(h)\supseteq E_8$, with $x_c$ and $x_e$ giving the relevant cut-points for $X^{(0)}$. Hence we immediately obtain that
\[\limsup_{h\rightarrow\infty}\frac{-\log \mathbb{P}(\mathcal{E}_j^{(i)}(h))}{h}\leq
\sum_{k=1}^8\lim_{\varepsilon\rightarrow0}\limsup_{h\rightarrow\infty}\frac{-\log p_k}{h},\]
where $p_1:=\mathbb{P}(E_1)$ and $p_k:=\mathbb{P}(E_k\:\vline\: E_{k-1})$ for $k=2,\dots,8$. Note that, from $x_a$, the process $X^{(0)}$ continues as under its original law, but conditioned so that $X^{(0)}\cdot\ell^{(i)}$ does not fall below $x_a\cdot\ell^{(i)}$. However, since the event $E_8$ is contained within the latter event, which has strictly positive probability, it will be enough to estimate the probabilities $(p_k)_{k=1}^8$ for $X^{(0)}$ under its original law (conditional on starting at $x_a$).

\begin{figure}[t]
\begin{center}
\includegraphics[scale=0.25]{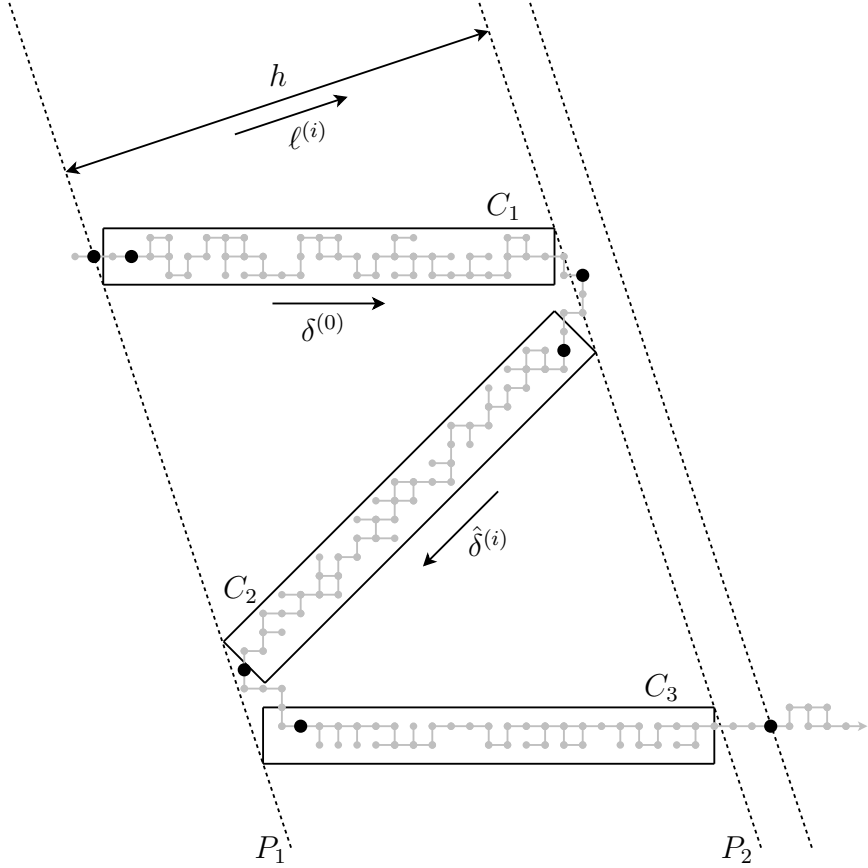}
\rput(-8.1,0){$P_1$}
\rput(-1.9,0){$P_2$}
\rput(-5,8.55){$C_1$}
\rput(-8.5,3.45){$C_2$}
\rput(-2.9,2.15){$C_3$}
\rput(-7.6,9.5){$\ell^{(i)}$}
\rput(-7.4,6.95){$\delta^{(0)}$}
\rput(-5.2,4.1){$\hat{\delta}^{(i)}$}
\rput(-8.0,10.3){$h$}
\end{center}
\caption{A possible realisation of the path of $X^{(0)}$ on the event $E_8$, as defined in the proof of Lemma \ref{trapprob}. The black dots mark the positions of $x_a$ (top left), $x_b$, $x_c$, $x_d$, $x_e$, $x_f$ and $x_g$ (bottom right).}
\label{trapfig}
\end{figure}

For $k=1,3,5,7$, since $E_k=E_{k-1}\cap E_k'$ where $E_k'$ only requires the random walk $X^{(0)}$ to follow an exact path of $\leq C\varepsilon h$ steps, it follows that
\[\lim_{\varepsilon\rightarrow0}\limsup_{h\rightarrow\infty}\frac{-\log p_k}{h}\leq \lim_{\varepsilon\rightarrow0}\left(-C\varepsilon \log \min_e p^{(0)}(e)\right)=0.\]
Next, note the strong law of large numbers and some basic calculations yields that, for $\eta>0$,
\[\lim_{n\rightarrow\infty}\mathbb{P}\left(\left|\frac{1}{n}X^{(0)}_{\lfloor nt\rfloor }-t{\delta}^{(0)}\right|\leq \eta\mbox{ for all }t\in[0,1]\right)=1.\]
Applying the change of parameter $n=h/({\delta}^{(0)}\cdot\ell^{(i)})$ and $\eta=\varepsilon\delta^{(0)}\cdot\ell^{(i)}$, the event within the above probability ensures that the process $(X^{(0)}_{\lfloor th/({\delta}^{(0)}\cdot\ell^{(i)})\rfloor})_{t\in[0,1]}$ stays within a distance $\varepsilon h$ of the linear function $(th\delta^{0}/({\delta}^{(0)}\cdot\ell^{(i)}))_{t\in[0,1]}$, which at time 1 in particular has traversed a distance $h$ in the $\ell^{(i)}$ direction. From this observation, we easily obtain that
\[\lim_{\varepsilon\rightarrow0}\limsup_{h\rightarrow\infty}\frac{-\log p_k}{h}=0, \quad \text{for }k=2,6.\]
Moreover, it is also straightforward to check from the directional transience of $X^{(0)}$ that $p_8\geq C$. Hence the previous limit also holds with $k=8$.

To complete the proof, it is thus sufficient to establish that
\begin{equation}\label{e4lim}
\lim_{\varepsilon\rightarrow0}\limsup_{h\rightarrow\infty}\frac{-\log p_4}{h}\leq t_{(i)}.
\end{equation}
In this direction, we start by appealing to a sample path large deviations principle of Mogul$'$ski\u\i \cite{Mogulskii} (see also \cite[Theorem 5.1.2 and the following remark (b)]{DZ}) to deduce that
\[\lim_{\eta\rightarrow0}\limsup_{n\rightarrow\infty}\frac{-\log\mathbb{P}\left(\left|\frac{1}{n}X^{(0)}_{\lfloor nt\rfloor}-t\hat{\delta}^{(i)}\right|\leq \eta\mbox{ for all }t\in[0,1]\right)}{n}=\Lambda(\hat{\delta}^{(i)}),\]
where
\[\Lambda(\hat{\delta}^{(i)}):=\sup_{x\in\mathbb{R}^d}\left(x\cdot\hat{\delta}^{(i)}-\log\mathbb{E}\left[\ee^{X_1^{(0)}\cdot x}\right]\right).\]
Thus, by considering the change of parameter $n=h/(-\hat{\delta}^{(i)}\cdot\ell^{(i)})$ and $\eta=\varepsilon(-\hat{\delta}^{(i)}\cdot\ell^{(i)})$, we obtain that
\[\lim_{\varepsilon\rightarrow0}\limsup_{h\rightarrow\infty}\frac{-\log p_4}{h}\leq\frac{\Lambda(\hat{\delta}^{(i)})}{-\hat{\delta}^{(i)}\cdot\ell^{(i)}},\]
and hence it will be enough for our purposes to show that the right-hand side here is equal to $t_{(i)}$. Now, observe that
\begin{eqnarray*}
\lefteqn{\left.\frac{\partial}{\partial x_j}\left(x\cdot\hat{\delta}^{(i)}-\log\mathbb{E}\left[\ee^{X_1^{(0)}\cdot x}\right]\right)\right|_{x=-t_{(i)}\ell^{(i)}}}\\
&=&\left.\left(\hat{\delta}^{(i)}_j-\frac{\mathbb{E}\left[\ee^{X_1^{(0)}\cdot x}X_1^{(0)}\cdot e_j\right]}{\mathbb{E}\left[\ee^{X_1^{(0)}\cdot x}\right]}\right)\right|_{x=-t_{(i)}\ell^{(i)}}\\
&=&\hat{\delta}^{(i)}_j-\frac{\hat{\delta}^{(i)}_j}{\varphi_{(i)}(t_{(i)})}\\
&=&0,
\end{eqnarray*}
and so $x=-t_{(i)}\ell^{(i)}$ is a stationary point for $x\cdot\hat{\delta}^{(i)}-\log\mathbb{E}[\ee^{X_1^{(0)}\cdot x}]$. Moreover, since the latter expression is strictly concave as a function of $x$ (see, for example the second lemma of \cite{Rothaus}), the value $x=-t_{(i)}\ell^{(i)}$ must be its unique maximiser. It follows that
\[\Lambda(\hat{\delta}^{(i)})=-t_{(i)}\ell^{(i)}\cdot\hat{\delta}^{(i)}-\log \varphi_{(i)}(t_{(i)})= -t_{(i)}\ell^{(i)}\cdot\hat{\delta}^{(i)},\]
which confirms \eqref{e4lim}.
\end{proof}

\section{The limiting path}\label{zinftysec}

To complete the article, in this section we prove Theorem \ref{thm:main3}, which we recall concerns the simplicity of $\mathcal{Z}^{(\infty)}$.

\begin{proof}[{\bf \em Proof of Theorem \ref{thm:main3}}(a)] For convenience, in this proof we will suppose that $\bs{p}^{(i)}$ is constant for all $i\geq 1$. To adapt the argument to the more general assumption of the theorem, one can just use an appropriate subsequence.

Let $(\mathcal{T}_j)_{j\geq 1}$ be the regeneration times of $X^{(0)}$ in direction $e_1$. Then, for every $i\geq 1$, $\mathcal{Z}^{(i)}$ necessarily contains each edge of the form $(X^{(0)}_{\mathcal{T}_{j}-1}, X^{(0)}_{\mathcal{T}_j})$ for $j\geq 2$. Now, conditional on $\mathcal{Z}^{(i)}$, the probability that the random walk $X^{(i)}$ started from $X^{(0)}_{\mathcal{T}_{j}}$ never hits $X^{(0)}_{\mathcal{T}_{j}-1}$ is given by
\[p_{i,j}:=\frac{r_{(i)}(X_{\mathcal{T}_{j}-1},X_{\mathcal{T}_{j}})}{R_{(i)}(X_{\mathcal{T}_{j}-1},\infty)}.\]
(To check this, one can apply \cite[Theorem 2.11]{Barlow}, for example.) We can bound the denominator above uniformly in $i$ by the almost-surely finite random variable that appears on the left-hand side of \eqref{pisum}. The numerator is constant by assumption. Hence $p_{i,j}\geq p_j$ for some $\mathbb{P}$-a.s.\ strictly positive random variable $p_j$.

Now, fix $j\geq 2$. Given $X^{(0)}$, the configuration of $\mathcal{Z}^{(i)}$ between the vertices $X_{\mathcal{T}_{j-1}}^{(0)}$ and $X_{\mathcal{T}_j}^{(0)}$ is given by one of a finite collection of graphs (each containing a simple path between $X_{\mathcal{T}_{j-1}}^{(0)}$ and $X_{\mathcal{T}_j}^{(0)}$). From this observation and the conclusion of the previous paragraph, we deduce that there exists a strictly positive random variable $\tilde{p}_j$ depending only on $X^{(0)}$ such that, conditional on $\mathcal{Z}^{(i)}$, the probability that the random walk $X^{(i)}$ started from $X^{(0)}_{\mathcal{T}_{j-1}}$ creates a simple path from $X_{\mathcal{T}_{j-1}}^{(0)}$ to $X_{\mathcal{T}_j}^{(0)}$, after which it never returns to $X^{(0)}_{\mathcal{T}_{j}-1}$, is bounded below $\tilde{p}_j$. Since, conditional on $\mathcal{Z}^{(i)}$, $X^{(i)}$ is independent of the earlier walks, it readily follows from the latter observation that $\mathbb{P}$-a.s.\ eventually one of the walks $X^{(i)}$ will create a simple path between $X_{\mathcal{T}_{j-1}}^{(0)}$ and $X_{\mathcal{T}_j}^{(0)}$, and clearly this will remain as the unique path connecting these two vertices in $\mathcal{Z}^{(\infty)}$.

Since the regeneration times of interest occur infinitely often along the path of $X^{(0)}$, we obtain that the part of the graph $\mathcal{Z}^{(\infty)}$ from $X^{(0)}_{\mathcal{T}_2-1}$ to $\infty$ is a simple path. Since $\mathcal{T}_2$ is a finite random variable, essentially the same argument yields that eventually one of the walks $X^{(i)}$ will create a simple path from $0$ to $X^{(0)}_{\mathcal{T}_2}$, and never return to $X^{(0)}_{\mathcal{T}_2-1}$. From this, we obtain the result.
\end{proof}

\begin{proof}[{\bf \em Proof of Theorem \ref{thm:main3}}(b)] Suppose $e\ne e_1$ is an edge satisfying the assumption of the relevant part of the theorem, and let $E=E_0$ be the event that $X^{(0)}_1=-e_1$, $X^{(0)}_2=0$, $X^{(0)}_3=e$, $X^{(0)}_4=0$, $X^{(0)}_5=e_1$, and $5$ is a regeneration time for $X^{(0)}$ in the direction $e_1$. Moreover, for $i\geq1$, let $E_i$ be the event that $E_j$ holds for $j<i$ and also $e\in\mathcal{Z}^{(i+1)}$, i.e.\ the vertex $e$ is visited by $X^{(i)}$. A standard result for random walks (apply \cite[Theorem 2.11]{Barlow} with $x=e$) then gives that
\begin{equation}\label{e1}
\mathbb{P}\left(E_i\:|\:E_{i-1},\:\mathcal{Z}^{(i)}\right)=1-\frac{r^{(i)}(0,e)}{R_{(i)}(0,\infty)+r^{(i)}(0,e)}\geq 1-\frac{r^{(i)}(0,e)}{R_{(i)}(0,\infty)}.
\end{equation}
Now, on $E_{i-1}$, we have that $R_{(i)}(0,\infty)\geq r^{(i)}(0,e_1)$, and so the above bound yields
\begin{equation}\label{e2}
\mathbb{P}\left(E_i\:|\:E_{i-1}\right)=1-\frac{c^{(i)}(0,e_1)}{c^{(i)}(0,e)}.
\end{equation}
On the other hand, similarly to the proof of Lemma \ref{lem:rec}, we have that
\[R_{(i)}(0,\infty)\geq r^{(i)}(0,e_1)\sum_{j=2}^\infty \beta_{(i)}^{-X_{\mathcal{T}_j}^{(0)}\cdot \ell^{(i)}},\]
where $(\mathcal{T}_j)_{j\geq 1}$ are the regeneration times of $X^{(0)}$ in direction $e_1$. Hence, using the obvious adaptation of the strong law of large numbers that appears at (\ref{llnx0regen}), we deduce that there exists a deterministic constant $c$ and finite random variable $N\geq2$ such that, $\mathbb{P}$-a.s.,
\[R_{(i)}(0,\infty)\geq r^{(i)}(0,e_1)\sum_{j=N}^\infty \beta_{(i)}^{-cj\delta^{(0)}\cdot \ell^{(i)}}=\frac{r^{(i)}(0,e_1)\beta_{(i)}^{-cN}}{1-\beta_{(i)}^{-c\delta^{(0)}\cdot \ell^{(i)}}}.\]
Thus, applying this bound in conjunction with \eqref{e1} and \eqref{e2} yields
\[\mathbb{P}\left(E_i\:|\:E_{i-1},\:N\right)=1-\frac{c^{(i)}(0,e_1)}{c^{(i)}(0,e)}\min\left\{1,\left(\beta_{(i)}^{c\delta^{(0)}\cdot \ell^{(i)}}-1\right)\beta_{(i)}^{cN}\right\}.\]
Since the event that ${Z}^{(\infty)}$ is not simple contains $\cap_{i\geq 0}E_i$, it follows that
\begin{eqnarray*}
\lefteqn{\P\left(\mc{Z}^{(\infty)} \text{ is not simple}\:\vline\: E_0,\:N\right)}\\
&\geq& \prod_{i=1}^\infty\left(1-\frac{c^{(i)}(0,e_1)}{c^{(i)}(0,e)}\min\left\{1,\left(\beta_{(i)}^{c\delta^{(0)}\cdot \ell^{(i)}}-1\right)\beta_{(i)}^{cN}\right\}\right),
\end{eqnarray*}
which is strictly positive whenever
\[\sum_{i=1}^\infty\frac{c^{(i)}(0,e_1)}{c^{(i)}(0,e)}\min\left\{1,\left(\beta_{(i)}^{c\delta^{(0)}\cdot \ell^{(i)}}-1\right)\beta_{(i)}^{cN}\right\}<\infty.\]
By the assumption of the lemma and the fact that $N$ is a finite random variable, the above sum is  $\mathbb{P}$-a.s.\ finite, and hence we have demonstrated that
\[\P\left(\mc{Z}^{(\infty)} \text{ is not simple}\right)\geq \mathbb{P}\left(E_0\right).\]
Finally, we note that it is easy to check that the right-hand side here is strictly positive, and so the proof is complete.
\end{proof}

\section*{Acknowledgements}
The work of MH was supported by Future Fellowship FT160100166, from the Australian Research Council. DC would like to thank the School of Mathematics and Statistics at the University of Melbourne for its generous support during a visit to Melbourne in August 2018, which is when the majority of the work on this article was completed, and also acknowledge the support of his JSPS Grant-in-Aid for Research Activity Start-up, 18H05832. MH thanks Ross Ihaka for providing the main R code used to perform simulations.

\bibliography{RWoRWoRW}
\bibliographystyle{amsplain}

\end{document}